\documentclass{amsart}
\usepackage{amsmath}
\usepackage[pagebackref=true]{hyperref}
\hypersetup{colorlinks=true,citecolor=blue,linkcolor=blue}
\newcommand{\Hilb}{{\rm{Hilb}}}

\newcommand{\image}{{\rm{Im}}\;}

\newcommand{\Pic}{{\rm Pic}}

\newcommand{\Spec}{{\rm{Spec}}\;}
\newcommand{\Hom}{{\rm{Hom}}}

\newcommand{\trace}{{\rm{trace}}\;}

\newcommand{\PP}{{\bf P}}
\newcommand{\AAA}{{\bf A}}
\newcommand{\CC}{{\mathbb C}}

\newcommand{\GG}{{\mathbb G}}

\newcommand{\ZZ}{{\mathbb Z}}

\def\Gr #1#2{{\mathbb G}(#1,#2)} 

\newtheorem{lemma}{\bf Lemma}[section]
\newtheorem{proposition}[lemma]{\bf Proposition}
\newtheorem{theorem}[lemma]{\bf Theorem}
\newtheorem{corollary}[lemma]{\bf Corollary}

\newtheorem{remark}[lemma]{\bf Remark}
\newtheorem{claim}[lemma]{\bf Claim}

\numberwithin{equation}{section}

\begin{document}

\title[Variety of Polar Simplices]{The Variety of Polar Simplices}

\author[Kristian Ranestad]{Kristian Ranestad}
\address{Matematisk institutt\\
         Universitetet i Oslo\\
         PO Box 1053, Blindern\\
         NO-0316 Oslo\\
         Norway}
\email{ranestad@math.uio.no}
\urladdr{\href{http://folk.uio.no/ranestad/}{http://folk.uio.no/ranestad/}}
\thanks{The first author was partially supported by an UCM-EEA grant under the 
NILS program, and both authors were supported by Institut
Mittag-Leffler.}
\author[Frank-Olaf Schreyer]{Frank-Olaf Schreyer}
\address{Mathematik und Informatik\\
Universit\"at des Saarlandes\\
 Campus E2 4\\
 D-66123 Saarbr\"ucken, Germany}
\email{schreyer@math.uni-sb.de } 
\urladdr{\href{http://www.math.uni-sb.de/ag/schreyer/}{http://www.math.uni-sb.de/ag/schreyer/}}

\date{\today}

\date{\today}

\begin{abstract}
A collection of n distinct hyperplanes $L_i = \{l_i=0\} \subset 
\PP^{n-1}$, the $(n-1)$-dimensional projective space over an algebraically closed field of characteristic not equal to $2$,
is a polar simplex of a smooth quadric $Q^{n-2}=\{q=0\}$, if each $L_i$ is
the polar hyperplane of the
point $p_i = \bigcap_{j \ne i} L_j$, equivalently, if $q=
l_1^2+\ldots+l_n^2$
for suitable choices of the linear forms $l_i$. In this paper we study
the closure $VPS(Q,n) \subset \Hilb_{n}(\check \PP^{n-1})$ of the 
variety of 
sums of powers presenting $Q$ from a global viewpoint:  $VPS(Q,n)$ is 
a smooth Fano variety of index $2$ and Picard number $1$ when $n<6$, 
and $VPS(Q,n)$ is singular when $n\geq 6$.
\end{abstract}

\maketitle
\tableofcontents

\section{Introduction} 

Let $Q=\{q=0\}$ be a $(n-2)$-dimensional smooth quadric defined over the complex numbers, or any algebraically closed field of characteristic not equal to $2$. 
We denote the
projective space containing $Q$ by $\check \PP^{n-1}$ because its 
dual space
$ \PP^{n-1} $ plays the major role in this paper. 
A collection $L_1 =\{ l_1=0 \}, \dots,
L_n= \{l_n=0\}$ of n hyperplanes is a polar simplex iff each $L_i$
is the polar of the point $p_i= \bigcap_{j\ne i} L_j$, equivalently,
iff the quadratic equation 
$$q=\sum_{i=1}^n l_i^2$$
holds for suitable choices of the linear forms $l_i$ defining $L_i$. 
In this paper we study the collection of polar simplices, or
equivalently, the variety of sums of powers presenting $q$ from a
global viewpoint. 

We may regard a polar simplex as a point in 
$\Hilb_{n}(\PP^{n-1})$.
Let $VPS(Q,n) \subset \Hilb_{n}(\PP^{n-1})$ be the closure of the
variety of sums of $n$ squares presenting $Q$. The first main result is:

\begin{theorem}\label{main} 
 If $2\leq n\leq 5$, then $VPS(Q,n)$ is a smooth rational $\binom{n}{2}$-dimensional Fano variety of index $2$ and Picard number $1$.
 If $n\geq 6$, then $VPS(Q,n)$ is a singular rational  $\binom{n}{2}$-dimensional variety.

\end{theorem}

If $n=2$, then $VPS(Q,n)=\PP^{1}$, and if $n=3$, then $VPS(Q,n)$ is a rational Fano threefold of index $2$ and degree $5$ (cf.\cite {Mu}).

The quadratic form  defines a collineation $q: \check \PP^{n-1}\to \PP^{n-1} $, let $q^{-1}: \PP^{n-1}\to \check \PP^{n-1}$ be the inverse collineation, and $Q^{-1}=\{q^{-1}=0\}\subset \PP^{n-1}$ the corresponding quadric.   
Consider the double Veronese embedding of $Q^{-1}\to  \PP^{\binom{n+1}{2}-2}$, and let $TQ^{-1}$ be the image by the Gauss map 
of tangent spaces  $Q^{-1}\to \GG(n-1,\binom{n+1}{2}-1)$.  Our second main result is: 

\begin{theorem}\label{main2} $VPS(Q,n)$ has a natural embedding in the Grassmannian variety $\GG(n-1, \binom{n+1}{2}-1)$ and contains the image  $TQ^{-1}$  of the Gauss map of the quadric $Q^{-1}$ in its Veronese embedding.
When $n=4$ or $n=5$  the restriction of the Pl\"ucker divisor generates the Picard group of $VPS(Q,n)$, and the degree
is $310$, resp. $395780$.
\end{theorem}

\medskip
We denote the coordinate ring  of $\PP^{n-1}$ by $S=\CC[x_{1},\ldots,x_{n}]$ and the coordinate ring of the dual $\check\PP^{n-1}$ by $T=\CC[y_1,\ldots,y_n]$.   In particular $S_{1}=(T_{1})^*$, so we may  set $\PP^{n-1}=\PP(T_{1}) $, the projective space of $1$-dimensional subspaces of $T_1$ with coordinate functions in $S$, and $\check\PP^{n-1}=\PP(S_{1})$ with coordinate functions in $T$.
Let $q \in T=\CC[y_1,\ldots,y_n]$ be a quadratic form defining the smooth $(n-2)$-dimensional quadric $Q\subset \check\PP^{n-1}=\PP(S_{1})$.  
Regard $[q]$ as a point in $\PP(T_{2})$ and consider the Veronese variety $V_{2}\subset \PP(T_{2})$ of squares, $$V_{2}=\{[l^2]\in\PP(T_{2})|l\in \PP(T_{1})\}. $$  
Then a polar simplex to $Q$ is simply a collection of $n$ points on $V_{2}$ whose linear span contains $[q]$.
Any length $n$ subscheme  $\Gamma\subset V_2$ whose span in  $\PP(T_{2})$ contains $[q]$ is called an  {\bf apolar subschemes of length $n$ to $Q$}.
The closure  $VPS(Q,n)$ of the polar simplices  in $\Hilb_{n}(\PP(T_1))$ consists of apolar subschemes of length $n$. 
We denote by $VAPS(Q,n)$ the subset of  $\Hilb_{n}(\PP(T_1))$, with reduced scheme structure, parameterizing  all apolar subschemes of  length $n$ to $Q$.  
Our third main result is:
\begin{theorem}\label{main3} The algebraic set $VAPS(Q,n)$ is isomorphic to the complete linear section 
$$VAPS(Q,n)=\langle TQ^{-1}\rangle\cap\GG(n-1, T_2/q)\subset \PP(\wedge^{n-1}(T_2/q))$$ 
in the Pl\"ucker space.   For $n\leq 6$ the two subschemes $VPS(Q,n)$ and $VAPS(Q,n)$ coincide.  For $n\geq 24$, the scheme $VAPS(Q,n)$ has more than one component.
\end{theorem}
Notice that we do not claim that the linear section $\langle TQ^{-1}\rangle\cap\GG(n-1, T_2/q)$ is reduced, only that its reduced structure coincides with $VAPS(Q,n)$.
The linear span $\langle TQ^{-1}\rangle$ has dimension $\binom{2n-1}{n-1}-\binom{2n-3}{n-2}-1$, while the Grassmannian has dimension $(n-1)\binom{n}{2}$ in $\binom{\binom{n+1}{2}-1}{n-1}$-dimensional Pl\"ucker space.  So this linear section is far from a proper linear section when $n\geq 4$, i.e. the codimension of $VAPS(Q,n)$ in the Grassmannian is much less than the codimension of its linear span in the Pl\"ucker space.

We find a covering of $VAPS(Q,n)$ by affine subschemes $V^{\it aff}_h(n)$ that are contractible to a point $[\Gamma_p]\in VPS(Q,n)$ (Lemma \ref{lem:weights}).   Therefore the apolar subschemes $\Gamma_p$  play a crucial point.  Let us explain what they are:
The projection of the Veronese variety $V_{2}\subset \PP(T_{2})$ from $[q]\in \PP(T_{2})$ is a variety $V_{2,q}\subset \PP(T_{2}/q)$.  The double Veronese embedding of $Q^{-1}$ is a linearly normal subvariety in $V_{2,q}$ that spans  $\PP(T_{2}/q)$.  For each point $p\in Q^{-1}$ consider the tangent space to $Q^{-1}$ in this embedding.  This tangent space intersects $V_{2,q}$ along the subscheme $\Gamma_p$.  

The affine subscheme $V^{\it aff}_h(n)$ is contractible to $\Gamma_p$, but depend only on a hyperplane:  It consists of the apolar subschemes that do not intersect a tangent hyperplane $h$ to $Q^{-1}$.  The point $p$ is simply a point on $Q^{-1}$ that does not lie in this hyperplane.

Our computations show that the affine scheme $V^{\it aff}_h(n)$ and certain natural subschemes has particularly interesting structure:
$V^{\it aff}_h(n)$ is isomorphic to an affine space when $n<6$ while $V^{\it aff}_h(6)$
is isomorphic to a $15$-dimensional cone over the $10$-dimensional spinor variety  (Corollary \ref{cor:n=6}).
Why this spinor variety appears is quite mysterious to us.   Recall that Mukai showed that a general canonical curve of genus $7$ is a linear section of the spinor variety.
Let $ V^{\it loc}_p(n)\subset VAPS(Q,n)$ be the subscheme of apolar subschemes in 
$VAPS(Q,n)$ with support at a single point $p\in Q^{-1}$.  
The subscheme  
$V^{\it loc}_p(n)$ is 
naturally contained in $V^{\it sec}_{p}(n)$, the variety of apolar subschemes in 
$V^{\it aff}_h(n)$ that contains the point $p$.  
We compute these subschemes with {\it Macaulay2} \cite {MAC2} when $n<6$ and find that $ V^{\it 
loc}_{p}(5)$ is isomorphic to a $3$-dimensional cone over the tangent developable of a rational normal sextic curve.  This  cone is a codimension $3$ linear section of  the scheme $V^{\it 
sec}_{p}(5)$, which is isomorphic to a $6$-dimensional cone over the intersection of  the Grassmannian $\GG(2,5)$ with a quadric.  
Mukai showed that a general canonical curve of genus $6$ is a linear section of the intersection of  $\GG(2,5)$ with a quadric.  
The appearances in the cases $n=5,6$ of a natural variety whose curve sections are canonical curves is both surprising and unclear to us.
The computational results are summarized in 
Table 1 in Section \ref{sec:affine}.




By the very construction of polar simplices, it is clear that $VPS(Q,n)$ has dimension $\binom{n}{2}$.  On the other hand, the special orthogonal group 
$SO(n,q)$ that preserves the quadratic form $q$, acts on 
the set of polar simplices: If we assume that the symmetric matrix of 
$q$ with respect to the variables in $T$ is the identity matrix, then 
regarding $SO(n,q)$ as orthogonal matrices 
the rows define a polar simplex.  Matrix multiplication therefore 
defines a transitive action of $SO(n,q)$ on the set of polar simplices.
By dimension count, this action has a finite stabilizer at a polar 
simplex.   This stabilizer is simply the group of even permutations of the rows.

 The linear representation of $SO(n,q)$ on $T_{2}$ decomposes 
$$T_{2}=\langle q\rangle\oplus T_{2,q},$$
where the hyperplane $\PP(T_{2,q})$ intersect the Veronese variety $V_2$ along the Veronese image of $Q^{-1}$. 
Therefore we may identify $T_{2}/q=T_{2,q}$ and 
 the projection from $[q]$: $\PP(T_{2})\dasharrow\PP(T_{2,q})$
 is an $SO(n,q)$-equivariant projection.
 $Q^{-1}\subset \PP(T_{2,q})$ is a closed orbit, and similarly the image $TQ^{-1}$ of the Gauss map is a closed orbit for the induced representation 
 on the Pl\"ucker space of  $\GG(n-1,T_{2,q})$.  The linear span of this image is therefore the projectivization of an irreducible representation of $SO(n,q)$.
 The set of polar simplices form an orbit for the action of  $SO(n,q)$, so the linear span of $VPS(Q,n)$ is also the projectivization of an irreducible representation of $SO(n,q)$.  
Therefore
$$VPS(Q,n)\subset\;  \langle TQ^{-1}\rangle\cap\; \GG(n-1,T_{2,q}).$$
We show that the intersection $\langle TQ^{-1}\rangle\cap \; \GG(n-1,T_{2,q})$ parameterizes all apolar subschemes of length $n$, hence Theorem \ref{main}.

The organization of the paper follows distinct approaches to $VPS(Q,n)$.
To start with we introduce the classical notion of apolarity and regard polar simplices as apolar subschemes in $\PP(T_{1})$ of length $n$ with respect to $q$. 
We use syzygies to characterize these subschemes among elements of the Hilbert scheme.   
In fact, polar simplices are characterized by their smoothness, the Betti numbers of their resolution, 
and their apolarity with respect to $q$. 
Allowing singular subschemes, we consider all apolar subschemes of length $n$.  We show in Section \ref{sec:apolar} that
these subschemes naturally appear in the closure $VPS(Q,n)$ of the set of polar simplices in the Hilbert scheme.  For $n>6$ there may be apolar subschemes of length $n$
that do not belong to the closure $VPS(Q,n)$ of the smooth ones.   In fact, we show in Section \ref{sec:apolar} that at least for $n\geq 24$, there are nonsmoothable apolar subschemes of length $n$, i.e. that $VPS(Q,n)$ is not the only component of $VAPS(Q,n)$.

  The variety $VPS(Q,n)$, in its embedding in $\GG(n-1,T_{2,q})$, has order one, i.e. through a general point in $\PP(T_{2,q})$
  there is a unique $(n-2)$-dimensional linear space that form the span of an apolar subscheme $\Gamma$ of length $n$.  This 
  is a generalization of the fact that a general symmetric $n\times 
  n$ matrix has $n$ distinct eigenvalues.  In
  Section \ref{sec:param} we use a geometric approach to characterize the generality assumption.  
  
  The fact that $VPS(Q,n)$ has 
 order one, means that it is the image of a rational map  
 $$\gamma: \PP(T_{2,q})\dasharrow \GG(n-1,T_{2,q}).$$
 In Section \ref{sec:tensor} we use a trilinear form introduced 
 by Mukai to give equations for the map $\gamma$. 
 With respect to the variables in $T$ we may 
 associate a symmetric matrix $A$ to each quadratic form $q'\in T_{2,q}$.  The Mukai 
 form associates to $q'$ a space of quadratic forms in $S_2$  that vanish on all the projectivized eigenspaces of the 
 matrix $A$.  For general $q'$ these quadratic forms generate the ideal of the unique common polar 
 simplex of $q$ and $q'$. This is Proposition \ref{ideal}. The Mukai form therefore defines the universal family of polar simplices, although it does not extend to the whole boundary.
  Common apolar subschemes to $q$ and $q'$, when
 $q'$ has rank at most $n-2$, form the exceptional locus of 
 the map $\gamma$.
    
 We do not compute the image of $\gamma$ in $\GG(n-1,T_{2,q})$. Instead we compute affine perturbations of $[\Gamma_p]$ in $\GG(n-1,T_{2,q})$ that correspond to apolar subschemes to $Q$. These perturbations form the affine subschemes $V^{\it aff}_h(n)$ that cover $VAPS(Q,n)$.
   In Section \ref{sec:affine}
we make extensive computations of these affine subschemes.  Each once of them is  contractible to a point $[\Gamma_p]$ on the
 subvariety $TQ^{-1}\subset VPS(Q,n)$.
  The question of smoothness of $VPS(Q,n)$ is reduced to a question of smoothness of the affine scheme $V^{\it aff}_h(n)$ at the point $[\Gamma_p]$.
  For $n\leq 5$ we show that such a point is smooth, while for $n\geq 6$, it is singular.  
 The main result of  Section \ref{sec:affine} is however Theorem \ref{main3}, that $VAPS(Q,n)$ is a linear section of the Grassmannian.

  
   In the final Section \ref{sec:geometry} we return to the geometry of $VPS(Q,n)$ and compute the degree by a combinatorial 
   argument for any $n$.  
   The Fano-index is computed using the natural  $\PP^{n-2}$-bundle on $VPS(Q,n)$, obtained by restricting
   the incidence variety over the Grassmannian, and its 
  birational morphism to $\PP(T_{2,q})$.

%
We thank Tony Iarrobino for sharing his insight on Artinian Gorenstein rings with us, and Francesco Zucconi for valuable comments on a previous version of this paper.

Let us briefly summarize the notation:
\begin{itemize}
\item $\CC$ denotes the field of complex numbers.
\item $q\in T_2$ is a non-degenerate quadratic form, and defines a collineation $q:S_1\to T_1$ and a linear form $q: S_2\to \CC$.
\item $Q$ is the quadratic hypersurface $\{q^{-1}=0\}\subset \PP(S_1)$.
\item $q^{-1}\in S_2$ is a quadratic form, that defines the collineation $q^{-1}:T_1\to S_1$  inverse to $q$ and a linear form $q^{-1}: T_2\to \CC$.
\item $Q^{-1}$ is the quadratic hypersurface $\{q^{-1}=0\}\subset \PP(T_1)$.
\item $q^\bot\subset S_2$ is the kernel of the linear form $q: S_2\to \CC$.
\item $T_{2,q}$ is the kernel $(q^{-1})^\bot$ of the linear form $q^{-1}:T_2\to \CC$
\item $\pi_q:\PP(T_2)\dasharrow \PP(T_{2,q})$ is the projection from $[q]\in \PP(T_2)$, and $V_{2,q}\subset  \PP(T_{2,q})$ is the image under this projection of the Veronese variety $V_2\subset \PP(T_2)$.
\end{itemize}

   \section{Apolar subschemes of length $n$}\label{sec:apolar}
  
 We follow the approach of \cite{RS}:   The apolarity action is 
 defined as the action of  
 $S=\CC[x_1,\ldots,x_n]$ as polynomial differential forms on 
 $T=\CC[y_1,\ldots,y_n]$ by setting $x_i = \frac \partial {\partial 
 y_i}$. This makes the duality between $S_{1}$ and $T_{1}$ explicit 
 and, in fact, defines a natural duality between $T_{i}$ and $S_{i}$.
 The form $q\in T_{2}$ define the smooth $(n-2)$-dimensional quadric 
 hypersurface  $$Q=\{\;[\sum a_{i}x_{i}]\;|\;(\sum a_{i}\frac \partial {\partial 
 y_i})^{2}(q)=0\}\subset\PP(S_{1}).$$
Apolarity defines a graded Artinian Gorenstein algebra associated to 
$Q$:
$$A^Q = \CC[x_1,\ldots,x_n]/(q^\bot)$$
where 
$$q^\bot = \{ D \in S_2=\CC[x_1,\ldots,x_n]_2 | D(q)=0 \}.$$

A subscheme $Y\subset \PP(T_{1})$ is {\bf apolar} to $Q$, or equivalently {\bf apolar} to $q$, if the space of quadratic forms in its ideal $I_{Y,2}\subset q^\bot$.
The apolarity lemma (cf.\cite{RS} 1.3) says that any smooth $\Gamma$,  $[\Gamma] \in 
\Hilb_{n}(\PP(T_{1}))$ is a polar simplex with respect to $Q\subset 
\PP(S_{1})=\check\PP^{n-1}$ if and only if $I_{\Gamma,2}\subset 
q^{\bot}\subset S_2$,  
i.e. $\Gamma$ is  apolar to $Q$. 
 We drop, for the moment, the smoothness criterium and consider any  $[\Gamma]\in 
\Hilb_{n}(\PP(T_{1}))$, such that $\Gamma$ is apolar to $Q$.  Notice that since $Q$ is nonsingular, $\Gamma$
is nondegenerate.  But more is known:
The following are  
the graded
Betti numbers of $A^{Q}$ and $\Gamma$, given in {\it Macaulay2} 
notation \cite{MAC2}.
  
\begin{proposition}\label{betti numbers} a) For a smooth quadric $Q 
\subset 
\check \PP^{n-1}$ 
the syzygies of the apolar Artinian Gorenstein ring $A^Q$ are
$$
\begin{matrix}
1 & - & \ldots & - & \ldots & - & - \cr
- & \frac{n-1}{n+1} \binom{n+2}{2} & \ldots &  \frac{k(n-k)}{n+1}
\binom{n+2}{k+1} & \ldots &  \frac {n-1}{n+1}\binom{n+2}{n} & -
\cr
- & - & \ldots & - & \ldots & - & 1 \cr 
\end{matrix}
$$
b) A zero-dimensional nondegenerate scheme $\Gamma \subset \PP^{n-1}$
of length $n$ has syzygies
$$
\begin{matrix}
1 & - & \ldots & - & \ldots & - & - \cr
- & \binom{n}{2} & \ldots & k\binom{n}{k+1} & \ldots & 
(n-1)\binom{n}{n} & -
\cr 
\end{matrix}
$$
\end{proposition}

\begin{proof} Eg. \cite{Be} and \cite{ERS} \end{proof}
    
\begin{corollary}\label{grassmannian embedding}
The natural morphism 
$$VAPS(Q,n)\to\Gr 
{\binom{n}{2}}{q^{\bot}};\quad \Gamma\mapsto I_{\Gamma,2}\subset q^{\bot}$$  is injective.
Equivalently, there is a natural injective morphism 
$$VAPS(Q,n)\to\Gr 
{n-1}{T_{2,q}};\quad \Gamma\mapsto I_{\Gamma,2}^{\bot}\subset T_{2,q}$$
 into the variety of 
$(n-2)$-dimensional subspaces of $\PP(T_{2,q})$ that intersect the projected Veronese variety $V_{2,q}$ in a scheme of length $n$.  
In particular, the Hilbert scheme and Grassmannian compactification in $\GG(n-1, T_{2,q})$ of the variety of polar simplices coincide.
\end{corollary}
\begin{proof} 
Apolarity defines a natural isomorphism $ q^{\bot}\cong T_{2,q}^*$.  Therefore the subspace $I_{\Gamma,2}\subset q^{\bot}$  
defines a $(n-1)$-dimensional subspace 
$I_{\Gamma,2}^{\bot}\subset T_{2,q}.$
The intersection  $\PP(I_{\Gamma,2}^{\bot})\cap V_{2,q}$ with the projected Veronese variety is precisely $\pi_{q}(\Gamma)$.  

The variety $VPS(Q,n)\subset \Hilb_{n}(\PP^{n-1})$ is 
the closure of the set of polar simplices inside the set of apolar 
subschemes of length $n$. The former set is irreducible, while the 
latter set is a closed variety defined by the 
condition that the generators of the ideal of the subscheme lie in $q^{\bot}$.  By Proposition
\ref{betti numbers}, the map  
$\Gamma\mapsto I_{\Gamma,2}\subset q^{\bot}$ extends to all of 
$VPS(Q,n)$ as an injective morphism.
\end{proof}

We relate apolarity with to polarity  with respect to a quadric hypersurface.
The classical notion of polarity is the composition of the linear map $q^{-1}$ with apolarity:
The polar to a point $[l]\in \PP(T_1)$ with respect to $Q^{-1}$ is the hyperplane $h_l=\PP(q^{-1}(l)^\bot)\subset\PP(T_1)$, 
where $$(q^{-1}(l))^\bot=\{l'\in T_1|l'(q^{-1}(l))=q^{-1}(l\cdot l')=0\}.$$  
In particular, the polar hyperplane to $l$ contains $l$ if and only if $q^{-1}(l^2)=0$, i.e. the point $[l]$ lies on the hypersurface $Q^{-1}$.

Let $\Gamma\subset \PP(T_1)$ be a length $n$ subscheme that contains $[l]$ and is apolar to $Q$.  The subscheme $\Gamma'\subset \Gamma$ residual to $[l]$ is defined by the quotient $I_{\Gamma'}=I_{\Gamma}\colon (l^{\bot})$.  Since $\Gamma$ is non degenerate, $\Gamma'$ spans a unique hyperplane.
This hyperplane is defined by a unique linear form $u'\in S_1$, and is characterized by the fact that $u'\cdot u(q)=u' q(u)=0$ for all $u\in l^{\bot}$, so it is the hyperplane  $\PP(q(l^{\bot}))$.  But  
$$l'\in q(l^{\bot})\Leftrightarrow 0=q^{-1}(l')l=q^{-1}(l\cdot l')=q^{-1}(l)l' \Leftrightarrow l'\in (q^{-1}(l))^{\bot},$$
so $\PP(q(l^{\bot}))$ is the polar hyperplane $\PP(q^{-1}(l)^\bot)$ to $[l]$ with respect to $Q^{-1}$.
Thus the subscheme $\Gamma'$ residual to $[l]$ in $\Gamma$ spans the polar hyperplane to $[l]$ with respect to $Q^{-1}$.  

\begin{lemma}\label{nonred} A component of an apolar subscheme has support on $Q^{-1}$ if and only if this component is nonreduced.
\end{lemma}
\begin{proof} If a component is a reduced point, the residual is contained in the polar hyperplane to this point, so by nondegeneracy the polar hyperplane cannot contain the point.
If a component is nonreduced, the residual to the point supporting the component lies in the polar hyperplane to this point, 
so the point is on $Q^{-1}$.
\end{proof}


Each component  $\Gamma_0$ of an apolar subscheme to $q$ is apolar to a quadratic form $q_0$ defined on the span of $\Gamma_0$ and uniquely determined as a summand  $q$.  This is the content of the next proposition.  

\begin{proposition}\label{ort}  Let $\Gamma=\Gamma_1\cup \Gamma_2$ be an apolar subscheme of length $n$ to $q$ that decomposes into two disjoint subschemes $\Gamma_1$ and $\Gamma_2$ of length $n_1$ and $n_2$. Let $U_1\subset T_1$ and $U_2\subset T_1$ be subspaces such that $\Gamma_i$ spans $\PP(U_i)$.  Then there is a unique decomposition $q=q_1+q_2$ with $q_i\in (U_i)^2$.  
 
 Furthermore, subschemes $\Gamma_1\subset \PP(U_1)$ and $\Gamma_2\subset \PP(U_2)$  of length $n_1$ and $n_2$ are apolar to $q_1$ and $q_2$ respectively, if and only if $\Gamma_1\cup\Gamma_2$ is apolar to $q$.\end{proposition}
\begin{proof}
Since $\Gamma$ is nondegenerate,  $T_1=U_1\oplus U_2$.  Let $U_i^{\bot}\subset S_1$ be the space of forms vanishing on $U_i$ via apolarity.  Then $U_1^\bot$ are natural coordinates on $\PP(U_2)$ and likewise, $U_2^\bot$ are natural coordinates on $\PP(U_1)$.  Let $I_1\subset (U_2^\bot)^2$ be the quadratic forms generating the ideal of $\Gamma_1$ in $\PP(U_1)$, and likewise  $I_2$ the quadratic forms generating the ideal of $\Gamma_2$ in $\PP(U_2)$. 
 Then $ I_1\oplus I_2\oplus (U_1^\bot)\cdot (U_2^\bot)\subset S_2$ is the space of quadratic forms in the ideal of $\Gamma$.

Consider the intersections, $q_2^{\bot}=q^{\bot}\cap (U_1^\bot)^2 $ and $q_1^{\bot}=q^{\bot}\cap (U_2^\bot)^2 $.  Since $q$ is non degenerate, $q^{\bot}$ does not contain either of the subspaces $(U_i^\bot)^2$.   Therefore $q_2^{\bot}$ is a codimension one subspace in $(U_1^\bot)^2$ and is apolar to a quadratic form $q_2\in (U_2)^2$, unique up to scalar. Similarly, $q_1^{\bot}$ is apolar to a unique quadratic from $q_1\in (U_1)^2$.
The space of quadratic forms $q_1^{\bot}\oplus q_2^{\bot}\oplus (U_1^\bot)\cdot (U_2^\bot)$ is contained in $q^\bot$ and is apolar to the subspace $\langle q_1,q_2\rangle\subset T_2$.  Therefore, there are unique  nonzero coefficients $c_1$ and $c_2$ such that $q=c_1q_1+c_2q_2$. Furthermore, each $\Gamma_i$ is apolar to $q_i, i=1,2$.
It remains only to show the last statement. 
Assume $\Gamma_1$  and $\Gamma_2$ are apolar to $q_1$ and $q_2$ respectively.  Then $\Gamma_1\cup\Gamma_2$ is non degenerate of length $n$.  Let $I_1\subset  ( U_2^{\bot})^2$ be the generators of the ideal of $\Gamma_1$ and $I_2\subset  ( U_1^{\bot})^2$ be the generators of the ideal of $\Gamma_2$.  Then the quadratic forms  in 
\[
I_1\oplus I_2\oplus (U_1)^{\bot}\cdot (U_2)^{\bot}
\]
all lie in the ideal of $\Gamma_1\cup\Gamma_2$.  The dimension of this space of quadratic forms is
\[
\binom{n_1}{2}+\binom{n_2}{2}+n_1\cdot n_2=\binom{n}{2},
\]
so they generate the ideal of $\Gamma_1\cup\Gamma_2$. Since all these forms are apolar to $q=q_1+q_2$,  the subscheme $\Gamma_1\cup\Gamma_2$ is apolar to $q$.
\end{proof}     
\begin{remark} By Proposition \ref{ort} the orbits of $SO(n,q)$ in $VASP(Q,n)$ are characterized by their components. 
\end{remark}
\noindent
We shall return to the set of local apolar subschemes $V^{\it loc}_p(n)$ supported at a point $p\in Q^{-1}$ in section \ref{sec:affine}.

Here we show that apolar subschemes of length $n$ to $q$  are all locally Gorenstein.
\begin{lemma}\label{gor}  Let $B$ be a local Artinian $\CC=B/m_B$-algebra of length $n$  and 
$\Phi:{\rm Spec} B \to  \AAA^{n-1} \subset \PP^{n-1}$ the reembedding
given by $\CC$-basis of $m_B$. The subscheme $\image\Phi$ is apolar to a full rank quadric if and only if $B$
is Gorenstein.
\end{lemma}
\begin{proof} Let $\phi: A=\CC[x_1,..,x_{n-1}] \to B$ be the ring 
    homomorphism corresponding to $\Phi$.  Thus $\phi$ is defined by an linear $k$-isomorphism $\phi_1=A_{\leq 1}=\langle 1,x_1,...,x_{n-1}\rangle\to B$.
    Let $\pi:B\to (0:m_B)$ be the projection onto the socle of $B$, let $\psi:(0\colon m_B)\to \CC$ be a linear form and consider the bilinear form 
    \[
 A\times A\xrightarrow{\phi\cdot\phi} B\xrightarrow{\pi} (0:m_B)\xrightarrow{\psi} \CC,
        \]
 where the first map is the composition of $\phi$ with multiplication.  This map extends to the tensor product  $A\otimes A$, and the restriction then to the symmetric part $$(A_{\leq 1})^2\subset A_{\leq 1}\otimes A_{\leq 1}$$ defines  
a linear form $$\beta_\psi: (A_{\leq 1})^2\to \CC$$ and an associated quadratic form $$q_\psi:A_{\leq 1}\to \CC.$$    Clearly the kernel of $\beta_\psi$ generate an ideal in $A$ that is apolar to $q_\psi$. On the other hand, $B$ is Gorenstein if and only if the socle is $1$-dimensional.  So for the lemma, it suffices to prove that $q_\psi$ is non degenerate, i.e. has rank $n$, if and only if the linear form $\psi$ is an isomorphism.

But $q_\psi$ is degenerate if and only if  the kernel of $\beta_\psi$ contains $x\cdot A_{\leq 1}$ for some nonzero element $x\in A_{\leq 1}$.  
Now, $\beta_\psi(x\cdot A_{\leq 1})=0$ if and only if $\phi(x)\cdot B\cap(0:m_B)\subset \ker\psi$.  Since $B$ is Artinian,   $\phi(x)\cdot B\cap(0:m_B)$ is a nonzero subspace of $(0:m_B)$, so it suffices to consider elements $x$, which map to the socle.  But then the kernel of $\beta_\psi$ contains $x\cdot A_{\leq 1}$ precisely when $x$ is in the kernel of $\psi$ and the lemma follows.
\end{proof}
\begin{corollary}\label{VAPS} $VAPS(Q,n)$ is reducible for $ n \geq 24$
\end{corollary}
\begin{proof}
Consider a general graded Artinian Gorenstein algebra $B$ of embedding 
dimension $e$ and socle in degree $3$.
The length of $B$ is $2e+2$. By the Macaulay correspondence \cite{Mac} such algebras are in bijection with homogeneous forms, up to scalars, of 
degree $3$ in $e$ variables, hence depends on
$\binom{e-1+3}{3}=(e+2)(e+1)e/6-1$ variables. The family of smoothable 
algebras have dimension at most $e(2e+2)-1$
So for $e+2>2\cdot 6$ a general algebra $B$ cannot be smoothable, for trivial 
reason. In particular, $e=11$ hence $n=24$ is enough.
\end{proof}

We do not believe the bound $n\geq 24$ is sharp.

\section{A rational parameterization}\label{sec:param}

In this section we show that through a general point in $\PP (T_{2,q})$ there is a unique $n$-secant $(n-2)$-space to the projected 
Veronese variety $V_{2,q}$.
Furthermore, we give a characterization of the points for which there are more than one, i.e. infinitely many  $n$-secant $(n-2)$-spaces to $V_{2,q}$.

 If we choose basis a for $T_{1}$ 
    such that the symmetric matrix associated to $q$ is the identity 
    matrix, then  
   the eigenvectors of the symmetric matrix associated to a general quadric $q'$ are distinct.
 Thus, the symmetric matrices associated to $q$ and $q'$ have a unique set of $n$ common $1$-dimensional eigenspaces.
    We formulate this geometrically.

\begin{proposition}\label{diag} 
Let $q, q^{\prime}\in T_{2}$ be two
general quadrics.  Then there exists a unique $n-$simplex
$\{L_{1},\ldots, L_{n}\}$ polar to both $q$ and $q^{\prime}$.
\end{proposition}
\begin{proof}  
By the above it suffices to show the relation between the collection of common eigenspaces of the associated symmetric matrices and the common simplex.
So we assume that 
$q, q^{\prime}$
are quadrics of rank $n$
and that
$$q=\sum_{i=1}^nl_{i}^2,\quad{\rm and}\quad
q^{\prime}=\sum_{i=1}^n\lambda_{i}l_{i}^2,$$
where the $\lambda_{i}$ are pairwise distinct coefficients and 
$L_{i}=\{l_{i}=0\},\quad i=1,\ldots, n$.  Let
$$q_i=\lambda_{i}q-q^{\prime},\quad i=1,\ldots , n.$$
Then the $q_i$ are precisely the quadratic forms of the pencil
generated by $q$ and $q^{\prime}$ that have rank less than
$n$.  Furthermore their rank is exactly $n-1$
since $\lambda_i \neq \lambda_j$ for $i \neq j$.
Therefore each $q_i\in (U_i)^2$ for a unique rank $n-1$ 
subspace 
$U_{i}\subset T_{1}$.
Then the intersection $\cap_{i\neq j}U_i$ is the 
$1$-dimensional subspace 
generated by the nonzero linear form
$l_{j}$. Therefore the forms $l_{i}$ are determined uniquely by the 
pencil
generated by $q$ and $q^{\prime}$.
\end{proof}
A precise condition for generality in the proposition is given by 
rank:
\begin{lemma}\label{max rank-curvilinear} A pencil of quadratic forms 
    in $n$ variables have a unique common apolar subscheme of length $n$ if and only if every quadric in the
pencil have rank at least $n-1$ and some,
    hence the general quadric has rank $n$.   Furthermore the unique apolar subscheme is
curvilinear.\end{lemma}
\begin{proof}
Let $$\langle q'+\lambda q\rangle_{\lambda\in \AAA_\CC^1}$$ be a pencil with discriminant $\Delta\subset \AAA_\CC^1$, a 
scheme of length $n$.  Consider the incidence
$$ \{(D,\lambda)|D( q'+\lambda q)=0\}\subset\PP (T_1)\times \AAA_\CC^1$$
with projections $p_T$ and $p_\CC$. Clearly the fibers of each 
projection 
are all linear.  Now as 
in the proof of the proposition, a general length $n$ subscheme of  
$p_T(p_\CC^{-1}(\Delta)$ is a common apolar subscheme to the pencil 
of quadratic forms.
Therefore, the common apolar subscheme is unique if and only if $p_T(p_\CC^{-1}(\Delta)$ is finite, i.e.  the corank of any quadric in $L$ is at most 
$1$.  In this case  both projections
restricted to the incidence are isomorphisms onto their images.  In particular the apolar subscheme is isomorphic to $\Delta$, so it is curvilinear.
\end{proof}

\begin{remark}\label{curv} The ideal of the curvilinear image $\Gamma$ of the map
$$\Spec \CC[t]/(t^{n}) \to \PP^{n-1}$$
$$ t \mapsto (1:t:t^2:\ldots:t^{n-1}),$$
 is generated by the $2 \times 2$ minors of
$$
\begin{pmatrix}
 x_1 & x_2 & \ldots & x_{n-1} & x_n \cr x_2 & x_3 & \ldots &
x_n & 0 \cr 
\end{pmatrix},
$$
so the  $\Gamma$ is apolar to the maximal rank quadric 
$$\sum_{k=1}^n y_ky_{n+1-k}.$$
\end{remark} 
This remark generalizes to a partial converse of Lemma \ref{max rank-curvilinear}.
\begin{lemma}\label{curvilinear-max rank}
Any curvilinear nondegenerate zero-dimensional subscheme $\Gamma 
\subset \PP^{n-1}$ of 
 length $n$ is apolar to a quadric 
$Q \subset \check \PP^{n-1} $ of maximal rank.
\end{lemma}
\begin{proof} 
Let $\Gamma$ be a nondegenerated  curvilinear subscheme 
with $r$ components of length $n_{1},\ldots,n_{r}$ such that 
$n_{1}+\ldots+n_{r}=n$.
Then $\Gamma$ is projectively equivalent to 
$\Gamma'=\Gamma_{1}\cup\ldots\cup\Gamma_{r}$, where $\Gamma_{i}$ is the image 
of $$\Spec \CC[t]/(t^{n_{i}}) \to \PP(\CC^{n_{i}})\subset
\PP(\CC^{n_{1}}\oplus\ldots\oplus\CC^{n_{r}})$$ $$ t \mapsto 
(1:t:t^2:\ldots:t^{n_{i}-1}),$$ 
where the nonzero coordinates in the image are 
$x_{i,1},\ldots,x_{i,n_{i}}$.
The ideal of $\Gamma'$ is generated by the $2 \times 2$ minors of the 
$r$ matrices 
$$
\begin{pmatrix}
 x_{(1,1)} & \ldots & x_{(1,n_{1}-1)} & 
 x_{(1,n_{1})}\cr
 x_{(1,2)} & \ldots & x_{(1,n_{1})} & 0 \cr 
\end{pmatrix}\quad \cdots\quad 
\begin{pmatrix}
 x_{(r,1)} & \ldots & x_{(r,n_{r}-1)} & x_{(r,n_{r})} \cr
 x_{(r,2)} & \ldots & x_{(r,n_{r})} & 0 \cr 
\end{pmatrix},
$$
and the products $$x_{(i,k_{i})}x_{(j,k_{j})}\quad{\rm for}\quad 1\leq i<j\leq r, 
1\leq k_{i}\leq n_{i}, 1\leq k_{j}\leq n_{j}.$$
So $\Gamma'$ is apolar to the maximal rank quadric 
$$\sum_{i=1}^{r}\sum_{k=1}^{n_{i}} y_{(i,k)}y_{(i,n_{i}+1-k)}.$$
\end{proof}

More important to us will be that rank $n$ quadrics have apolar subschemes of length $n$ that are not curvilinear (when $n>3$).
 \begin{remark}\label{gammap}   Consider the rank $n$ quadric
 $$ q=2y_{1}y_{n}+y_{2}^2+\ldots +y_{n-1}^2.$$
 The subscheme $\Gamma_{p}\subset \PP^{n-1}$ defined by
$$(x_{1}^2, x_{1}x_{2}, x_{2}^2-x_{1}x_{n}, x_{1}x_{3}, \ldots ,x_{n-1}^2-x_{1}x_{n}),$$
has degree $n$ and is apolar to $q$, but it is clearly not curvilinear when $n>3$.  It contains
 the tangency locus of the quadric $\{q^{-1}=\frac{1}{2}x_{1}x_{n}+\frac{1}{4}x_{2}^2+\ldots +\frac{1}{4}x_{n-1}^2=0\}$ at the point $[0:0:...:1]$.  The tangency locus has length $n-1$  and is  defined by
 $$(x_{1}, (x_{2}^2-x_{1}x_{n},x_{2}x_{3},..., x_{n-1}^2-x_{1}x_{n})).$$
 The subscheme $\Gamma_{p}$  is itself not contained in the tangent hyperplane $\{x_{1}=0\}$, but it is the unique apolar subscheme to $q$ that contains  the first order neighborhood of  $[0:0:...:1]$ on $\{q^{-1}=0\}$.
 It will be the focus of our attention in Section \ref{sec:affine}.
\end{remark}

It follows immediately from Proposition \ref{diag} that there is a rational and dominant map  
$$\gamma: \PP(T_{2,q})\dasharrow VPS(Q,n)\subset\GG(n-1,T_{2,q})$$
whose general fiber is a $n$-secant $(n-2)$-space to the projected Veronese variety  $V_{2,q}$.
 In the next section we find equations for this map.

\section{The Mukai form}\label{sec:tensor}

Mukai introduced in \cite{Mu} a trilinear form in his approach to varieties of sums of powers of conics in particular, 
and to forms of even degree in general  (cf. \cite{Do} for a nice exposition).  
In this section we show how this form naturally gives equations for the map $\gamma$ and for the universal family of polar simplices.
The main result of this section,  Proposition \ref{ideal}, gives the equations for the common apolar subscheme of length $n$ of a pencil of quadrics in $n$ variables, whenever this subscheme is unique, cf. Lemma \ref{max rank-curvilinear}.

Both the quadratic form $q\in T_{2}$ and the inverse $q^{-1}\in S_{2}$ 
play a crucial role in the definition of the Mukai form. Recall that the form
$q$ defines an invertible linear map $q: S_{1}\to T_{1}$, and $q^{-1}$ defines the inverse map:  $q^{-1}:T_{1}\to S_{1}.$  
In coordinates we get that if
$q=(\alpha_{1}y_{1}^2+...+\alpha_{n}y_{n}^2)$, then $q^{-1}=(\frac{1}{4\alpha_{1}}x_{1}^2+...+\frac{1}{4\alpha_{n}}x_{n}^2)$.
We will arrive at Mukai's form from
$$ 
\tau \in \Hom(\wedge^2 S_1 \otimes T_2  \otimes T_2 \otimes S_2, \CC)
$$
defined by 
$$
z_1 \wedge z_2 \otimes q_1 \otimes q_2 \otimes \alpha \mapsto (z_1(q_1)z_2(q_2)-z_2(q_1)z_1(q_2))(\alpha)
$$
where $f(g)$, as above,   means $f$ viewed as differential operator applied to $g$. 
Interpreting $\omega=z_1 \wedge z_2  \in \wedge^2 S_1 \subset \Hom(T_1,S_1), q_j \in T_2 \subset \Hom(S_1,T_1)$  and  $\alpha \in S_2 \subset \Hom(T_1,S_1)$ the expression
\begin{align*}
(\omega \otimes q_1 \otimes q_2 \otimes \alpha) &\mapsto \frac{1}{2} \trace(\alpha \circ q_2\circ\omega \circ q_1-\alpha \circ q_1 \circ \omega \circ q_2)\\
&= \frac{1}{2} \trace(\omega \circ q_1 \circ \alpha \circ q_2-\alpha \circ q_1 \circ \omega \circ q_2)
\end{align*}
gives an alternative description of $\tau$, since $f(g)= \frac{1}{2} (\trace f \circ g)$ holds for $f \otimes g \in S_2 \otimes T_2 \subset \Hom(T_1,S_1) \otimes \Hom(S_1,T_1)$ and $\trace((\alpha \circ q_2)\circ(\omega \circ q_1))=\trace((\omega \circ q_1)\circ(\alpha \circ q_2)).$
 We now substitute $\alpha=q^{-1}$. Then $\frac{1}{2} \trace( \omega \circ q_1 \circ q^{-1} \circ q_2-q^{-1} \circ q_1 \circ \omega \circ q_2)=0 $ for $q_1=q$, and since the first expression for $\tau$ is alternating on $T_2 \otimes T_2$  we have $\frac{1}{2} \trace( \omega \circ q_1 \circ q^{-1} \circ q_2-q^{-1} \circ q_1 \circ \omega \circ q_2)=0 $ for $q_2=q$ as well. Thus $\tau$ induces a well defined
trilinear form 
$$
 \tau_q \in \Hom(\wedge^2 S_1 \otimes T_{2,q}  \otimes T_{2,q}, \CC)$$ on the quotient space $T_{2,q}=T_2/\langle q \rangle$. Since $T_{2,q}^* = q^{\bot}  \subset S_2=T_2^*$ and
$$
\Hom(\wedge^2 S_1 \otimes T_{2,q}  \otimes T_{2,q}, \CC) \cong  \Hom(T_{2,q}, \Hom(\wedge^2 S_1, q^{\bot}))
 $$ 
 we have a second interpretation of $\tau_q$. With this interpretation,
the image of $\tau_q(q_1) \in \Hom(\wedge^2 S_1,q^{\bot}) \subset \Hom(\wedge^2 S_1,S_2)$ is defined by
$$\: \omega \mapsto [\omega\circ q_1 \circ q^{-1}- q^{-1}\circ q_1 \circ \omega] \in q^{\bot} \subset S_2 \subset \Hom(T_1,S_1).$$
The form $\tau$ is alternating on $T_2\otimes T_2$, so $\tau(\omega,q',q',q^{-1})=0$ for every $\omega\in \wedge^2S_{1}$. 
Therefore $\tau_{q}(q')(\wedge^2S_{1})\subset (q')^{\bot}$.
If $Q'$ is the quadric $\{q'=0\}\subset\PP(S_{1})$, we may 
therefore conclude: 
\begin{lemma}\label{apolartoboth} Any quadratic form 
in $\tau_{q}(q')(\wedge^2S_{1})$ is apolar to both $Q$ and $Q'$:
$$\tau_{q}(q')(\wedge^2S_{1})\subset q^{\bot}\cap(q')^{\bot}.$$
    \qed\end{lemma}
    
 Notice that the linear space of quadratic forms 
 $\tau_{q}(q')(\wedge^2S_{1})$ is not all of 
 $q^{\bot}\cap(q')^{\bot}$.  It is a special subspace of the 
 intersection.
   Since $\tau_{q}(q)=0$,  we have $\tau_{q}(q')=\tau_{q}( 
   q'+\lambda q)$  for any $\lambda$, so the space $\tau_{q}(q')(\wedge^2S_{1})$ of quadratic forms
  depends only on the pencil $\langle q,q'\rangle$.
 
If the pencil of quadratic forms  $\langle q,q'\rangle\subset T_{2}$
contains no forms of corank at least $2$, then, by Lemma \ref{max rank-curvilinear}, there is a  
unique common apolar subscheme $\Gamma_{q'}$ of length $n$ to $q$ and $q'$.  
The significance of the form $\tau_q$ is 
\begin{proposition}\label{ideal}
Let $q'\in T_{2,q}$.
 Then the linear map 
$$\tau_q(q'):\wedge^2S\to q^{\bot}$$ 
is injective if and only if $q$ and $q'$ have a unique common apolar subscheme of length $n$.  Furthermore, in this case the image generates the ideal in $S$ of this subscheme.
\end{proposition}
\begin{proof}
Our argument depends on several lemmas, in which we study $\image \tau_{q}(q')\subset S_{2}$ by considering the symmetric matrices 
associated to these quadratic forms  with respect to a 
suitable basis. Thus, we choose 
coordinates such that $q=\frac{1}{2}(y_{1}^2+y_{2}^2+...+y_{n}^2)$ and 
hence $q^{-1}=\frac{1}{2}(x_{1}^2+x_{2}^2+...+x_{n}^2)$. 
The symmetric matrices of these quadratic forms with respect to 
the coordinate basis of $T_{1}$ and $S_{1}$ are both the identity matrix. 
We denote by $A$ the symmetric matrix of $q'$, i.e.  $q'=\frac{1}{2}(y_{1},...,y_{n})A(y_{1},...,y_{n})^t$.  
For a form $\omega\in \wedge^2S_{1}$ there is similarly an associated skew symmetric matrix $\Lambda_{\omega}$.
For a form $l\in T_{1}$ we denote by $v_{l}$ the column vector of its coordinates.
The quadratic forms in the image $\tau_q(q')$ are the forms 
associated to the symmetric bilinear forms 
$$\{\omega\circ q'\circ q^{-1}-q^{-1}\circ q'\circ \omega |\omega\in 
\wedge^2S_{1}\},$$ 
so their associated symmetric matrices
are $$\{\Lambda_{\omega}A-A\Lambda_{\omega}|\omega\in \wedge^2S_{1}\}.$$ 
\begin{lemma}\label{vanish}
Let $[l]\in \PP (T_{1})$, then every quadric in $\tau_q(q')(\wedge^2S)\subset q^{\bot}$ vanishes at the point $[l]$ if and only if there is a quadric $q_{\lambda}=q'+\lambda q$ for some $\lambda \in \CC$, such that
$l$ lies in the kernel of the linear transformation $q_{\lambda}\circ q^{-1}:T_{1}\to T_{1}.$ 

Equivalently, in terms of matrices: If $v_{l}$ is the column coordinate vector of $l$, then  $v_{l}^t(\Lambda_{\omega}A-A\Lambda_{\omega})v_{l}=0$ for every $\omega\in \wedge^2S_{1}$ if and only if $v_{l}$
is an eigenvector for the matrix $A$.
\end{lemma}
\begin{proof}
Note first that the matrix of the linear transformation 
$q_{\lambda}\circ q^{-1}$ , with respect to the coordinate basis of 
$T_{1}$, is simply $A+\lambda I$. Hence, the equivalence of the two parts of the lemma.

In the matrix notation, if $v_{l}$ is an eigenvector for $A$ with eigenvalue $\lambda$, then 
 $$v_{l}^t(\Lambda_{\omega}A-A\Lambda_{\omega})v_{l}=v_{l}^t\Lambda_{\omega}Av_{l}-v_{l}^tA\Lambda_{\omega}v_{l}$$
 $$=v_{l}^t\Lambda_{\omega}\lambda v_{l}-\lambda v_{l}^t\Lambda_{\omega}v_{l}=0,$$
 so the if part follows.  
 
 Conversely, assume that  $$v_{l}^t(\Lambda_{\omega}A-A\Lambda_{\omega})v_{l}=0$$ for every skew symmetric $n\times n$ matrix $\Lambda_{\omega}$.
Again
 $$v_{l}^t(\Lambda_{\omega}A-A\Lambda_{\omega})v_{l}=v_{l}^t\Lambda_{\omega}Av_{l}-v_{l}^t\Lambda_{\omega}Av_{l},$$ and since
 $A$ is symmetric and $\Lambda_{\omega}$ is skewsymmetric, 
 $(v_{l}^t\Lambda_{\omega}Av_{l})^t=-v_{l}^tA\Lambda_{\omega}v_{l}$, 
 so we deduce 
 that
 $$v_{l}^t\Lambda_{\omega}Av_{l}=0.$$
But $v_{l}^t\Lambda_{\omega}u=0$ for every skew symmetric matrix $\Lambda_{\omega}$ only
if $u$ is proportional to $v_{l}$, so we conclude that $A(v_{l})=\lambda v_{l}$ for some $\lambda$. 
\end{proof}
\begin{remark}\label{singq} A point $l\in T_{1}$ lies in the kernel of $q_{{\lambda}}\circ q^{-1}$ 
    if and only if $q^{-1}(l)$ lies in the kernel of $q_{\lambda}:S_{1}\to T_{1}$.
Equivalently, $\{q_{\lambda}=0\}\subset \PP(S_{1})$ is a singular quadric and $[q^{-1}(l)]\in \PP(S_{1})$ lies in its singular locus.
\end{remark}
\begin{corollary}  $\tau_q(q')$ is injective only if $\langle q,q'\rangle$ contains no quadratic form of rank less than $n-1$.
\end{corollary}
\begin{proof} If  the quadratic form $q_{\lambda}=q'+\lambda q$ has rank less than $n-1$, then there are independent forms $l,l'\in T_1$ such that $\langle q^{-1}(l),  q^{-1}(l')\rangle$ is contained in the kernel of $q_{\lambda}:S_{1}\to T_{1}$.  In particular, viewed as differential operators applied to $q_\lambda$, 
\[
q^{-1}(l)(q_\lambda)=q^{-1}(l')(q_\lambda)=0\in T_1.
\]
Let 
\[
\omega=q^{-1}(l)\wedge q^{-1}(l')\in \wedge^2 S_1\subset {\rm Hom}(T_1,S_1).
\]
Then 
\[
\omega\otimes q_\lambda\otimes q_2\otimes q^{-1} \mapsto [q^{-1}(l)(q_\lambda)\cdot q^{-1}(l')(q_2)-q^{-1}(l')(q_\lambda)\cdot q^{-1}(l)(q_2)] (q^{-1})=0
\]
for every $q_2\in T_2$, so $\tau_q(q')(\omega)=0$ and $\tau_q(q')$ is not injective.
\end{proof}
To complete the proof of Proposition \ref{ideal}, we assume that $q$ and $q'$ have a unique common apolar subscheme $\Gamma$ of length $n$, i.e. by Lemma \ref{max rank-curvilinear}, no quadratic form in $\langle q,q'\rangle$ has rank less than $n-1$.  We want to show that $\tau_q(q')$ is injective and that the image generates the ideal of $\Gamma$.

  Let $\Gamma=\Gamma_1\cup...\cup \Gamma_r$ be a decomposition of $\Gamma$ into its connected components.  Then each $\Gamma_i$ is a finite local curvilinear scheme.  Let $n_i$ be the length of $\Gamma_i$. 
By Proposition \ref{ort} there is a decomposition $T_{1}=\oplus_{i}U_{i}$ such that each $\Gamma_i\subset \PP(U_i)$.  Furthermore $U_i$ has dimension $n_i$ and $q$ and $q'$ have unique decompositions  $q=q_1+...+q_r$ and $q'=q'_1+...+q'_r$ with $q_i,q'_i\in (U_i)^2\subset T_2$.  
 Denote by $U'_i=\oplus_{j\not=i}U_j$, and let $(U_i')^{\bot}$ be the orthogonal subspace of linear forms in $S_1$.   Then 
 \[
 \sum_i U_{i}^{\bot}\cdot (U_i')^{\bot}\subset S_2,
 \]
 generate the ideal of $\cup_i\PP(U_i)\subset\PP(T_1).$  
 
 The linear forms $q^{-1}(U_i)\subset S_1$ are natural coordinates on $\PP(U_i)$.  Denote by $I_{\Gamma_i,2}$ the quadratic forms in these coordinates in the ideal of $\Gamma_i$.  Then $I_{\Gamma_i,2}\subset (q^{-1}(U_i))^2\subset S_2$ and the space of quadratic forms in the ideal of $\Gamma$ is
 \[
 I_{\Gamma,2}=\sum_i I_{\Gamma_i,2}+ \sum_i U_{i}^{\bot}\cdot (U_i')^{\bot}\subset S_2.
 \]
 
 We
 \begin{claim}\label{claim} \[ \image \tau_q(q')\supset \sum_i I_{\Gamma_i,2}+ \sum_i U_{i}^{\bot}\cdot (U_i')^{\bot}.
 \]
 \end{claim}
If the claim holds, $\tau_q(q')$ is injective, since ${\rm dim} \wedge^2 S_1={\rm dim}  I_{\Gamma,2}\quad(=\binom{n}{2})$, so the equality 
 $\image \tau_q(q')=I_{\Gamma,2}$
 holds  and  the proof of Proposition \ref{ideal} is complete.
%

We use matrices to prove the claim.
To interpret the decomposition of $q$ and $q'$ in terms of matrices, we choose a basis for each $U_i$ such that the symmetric matrix associated to each $q_i$ is the $n_i\times n_i$ identity matrix.
Let $A_i$ be the symmetric $n_i\times n_i$ matrix associated to $q'_i$.  The union of the bases for the $U_i$ form a basis for $T_1$
with respect to which the symmetric matrix $A$ of $q'$ has $r$ diagonal blocks $A_i$ and zeros elsewhere.

The matrices $A_i$ each have a unique eigenvalue $\lambda_i$, and these eigenvalues are pairwise distinct.  Furthermore, each $A_i$ has a $1$-dimensional eigenspace, so their Jordan form has a unique Jordan block, and we may write $A_i= \lambda_i I_{n_i} + B_i$ with $B_i$ a nilpotent symmetric matrix. (See
(\cite[Theorem 2.3]{DZ}) for a  nice normal form for the matrices $B_i$.)

 

 By extending each $A_i$ with zeros to $n\times n$ matrices we may write   $A=\sum A_i$.
   The decomposition $T_{1}=\oplus_{i}U_{i}$ 
is then defined by $U_i=\ker(\lambda_i I-A)^{n_i}\subset T_1$.
 
 

Denote by $U'_i=\oplus_{j\not=i}U_j$.  
    Then $\PP(U_i)$ and $\PP(U'_i)$ have complementary dimension in $\PP(T_{1})$.
We shall use the techniques applied by Gantmacher in the analysis of commuting matrices (\cite[Chapter VIII]{Ga}) to show

\begin{lemma}\label{tri1} Let  $A$ be the symmetric matrix of the 
    quadratic form $q'\in T_{2,q}$ as above. 
Let $T_1=U_{i}\oplus U'_{i}$ be the decomposition associated to the eigenvalue $\lambda_i$.
Then 
\[
U_{i}^{\bot}\cdot (U_i')^{\bot}\subset \image\tau_{q}(q')\subset S_2.
\]   
    \end{lemma}
    \begin{proof}  
   Set $d=n_i$ and $\lambda=\lambda_i$ and choose coordinates such that $U_{\lambda}=\langle y_{1},...,y_{d}\rangle$ and $U_{\lambda}'= \langle y_{d+1},...,y_{n}\rangle$.  Then $(U'_{\lambda})^{\bot}=\langle x_{1},...,x_{d}\rangle$ and $(U_{\lambda})^{\bot}= \langle x_{d+1},...,x_{n}\rangle$.
Consider 
the matrix $B$ of the quadratic form $\tau_{q}(q')(x_{i}\wedge x_{j})$ with $i\leq d$ and $j>d$.  The 
skew symmetric matrix $\Lambda_{(ij)}$ of $x_{i}\wedge x_{j}$ has 
$(ij)$-th entry $1$, consequently $(ji)$-th
entry $-1$, and $0$ elsewhere, and
$$B=\Lambda_{(ij)} A-A \Lambda_{(ij)}.$$
The nonzero entries in $\Lambda_{(ij)} A$ are in positions $(i,k)$ with $k>d$ and $(j,k)$ with $k\leq d$, while the nonzero entries in $A \Lambda_{(ij)}$ are in  positions $(k,i)$ with $k>d$ and $(k,j)$ with $k\leq d$.
Therefore the quadratic form $\tau_{q}(q')(x_{i}\wedge x_{j})$  lies in the space 
\[
\langle x_ax_b\rangle_{a\leq d<b}=(U'_{\lambda})^{\bot}\cdot (U_{\lambda})^{\bot}
\]


A linear relation between these quadratic forms would correspond to a skew symmetric matrix $\Lambda$ with nonzero entries
 only in the rectangular block $(ij), i\leq d, j> d$, such that $\Lambda A-A \Lambda =0$.  
Write $A$ as a sum $A=A_\lambda+A_{\mu_{1}}+...+A_{\mu_{s}}$ where the $\mu_{i}$ 
are the eigenvalues of $A$ distinct from $\lambda$.
Let $\Lambda$ be a skew symmetric matrix and let $\Lambda_{\lambda,\mu_{i}}$ 
be the rectangular submatrix with rows equal to the nonzero rows of $A_{\lambda}$  and columns equal to the nonzero columns of $A_{\mu_{i}}$.
Then the corresponding submatrix   
$$(\Lambda A-A \Lambda)_{\lambda,\mu_{i}}=
\Lambda_{\lambda,\mu_{i}}A_{\mu_{i}}-A_{\lambda}\Lambda_{\lambda,\mu_{i}}.$$ 
So $\Lambda A-A \Lambda=0$ only if $\Lambda_{\lambda,\mu_{i}}A_{\mu_{i}}-A_{\lambda}\Lambda_{\lambda,\mu_{i}}=0$ for each $\mu_{i}$.

Let $\mu$ be one of the $\mu_{i}$, and assume for simplicity $U_{\mu}=\langle y_{d+1},...,y_{d+e}\rangle$.
Let $I_{d}$ be the diagonal matrix with $1$ in the $d$ first entries and $0$ elsewhere, 
and let $I_{e}$ be the diagonal matrix with $1$ in the entries $d+1,...,d+e$ and $0$ elsewhere. 
Then the special summand $A_{\lambda}$ of $A$ can be written as  a sum $A_\lambda=\lambda I_d+B_{d}$ 
where $B_d$ is nilpotent of order $d$.  Likewise,  $A_{\mu}=\mu I_{e}+B_{e}$ 
 where $B_{e}$ is nilpotent of order $e$.  
 So we may write $A=\lambda I_d+B_{d}+\mu I_{e}+B_{e}+A'$, where $A'=A-A_{\lambda}-A_{\mu}$.  But then
 $(\Lambda A-A \Lambda)_{\lambda,\mu}=0$ only if 
  $$\Lambda_{\lambda,\mu} A_{\mu}-A_{\lambda} \Lambda_{\lambda,\mu}=0,$$
  i.e. when
  $$\Lambda_{\lambda,\mu}(\mu I_e+B_{e})-(\lambda I_d+B_{d}) \Lambda_{\lambda,\mu}=0.$$
  This is equivalent to
  $$(\lambda-\mu)\Lambda_{\lambda,\mu}=\Lambda_{\lambda,\mu}B_e-B_d \Lambda_{\lambda,\mu}.$$

 Multiplying both sides by $(\lambda-\mu)$ and substituting on the right hand side $(\lambda-\mu)\Lambda_{\lambda,\mu}$ with $\Lambda_{\lambda,\mu}B_e-B_d \Lambda_{\lambda,\mu}$ we get
$$(\lambda-\mu)^2\Lambda_{\lambda,\mu}=(\Lambda_{\lambda,\mu}B_e-B_d\Lambda_{\lambda,\mu})B_e-B_d(\Lambda_{\lambda,\mu}B_e-B_d \Lambda_{\lambda,\mu})$$
$$=(\Lambda_{\lambda,\mu} (B_e)^2-2B_d\Lambda_{\lambda,\mu} B_e+(B_d)^2\Lambda_{\lambda,\mu}).$$

Iterating $m=d+e-1$ times we get

$$(\lambda-\mu)^m\Lambda_{\lambda,\mu}=\sum_{s+t=m}(-1)^s\binom{s+t}{t}(B_d)^s\Lambda_{\lambda,\mu} (B_e)^t$$
But on the right hand side either $(B_d)^s=0$ or $(B_e)^t=0$ when $s+t=d+e$, so $\Lambda_{\lambda,\mu}=0$.

 Thus $\Lambda_{\lambda,\mu_{i}}=0$ for all $i$, and the symmetric 
 matrices $A\Lambda_{ij}-\Lambda_{ij}A$ with $i\leq d, j>d $ are linearly independent.
   The corresponding quadratic forms therefore are linearly independent in the space $\langle x_1,...,x_d\rangle\times\langle x_{d+1},...,x_{n}\rangle.$ 
   Since the dimensions coincides, the quadratic forms span this space, and the lemma follows.
\end{proof}

Next, we consider the case when the symmetric matrix $A$ only has one eigenvalue. 
Thus we assume that $\Gamma$ has only one component, the symmetric matrix $A$  of $q'$ has only one eigenvalue and up to scalars only one nonzero eigenvector. Hence $\langle q,q'\rangle$ contains exactly one quadratic form of rank $n-1$.  In particular, by Lemma \ref{max rank-curvilinear}, $\Gamma$ is curvilinear.
Without loss of generality we may assume that $q'$ has rank $n-1$ i.e. that the eigenvalue is $0$.  Then $A$ is nilpotent, and since $A$ is a one-dimensional eigenvector space, $A^n=0$ and $A^i\not= 0$ for any $i<n$.    


\begin{lemma}\label{ideallocal}  Let $q'\in T_{2,q}$ be a quadratic 
    form whose associated $n\times n$ matrix $A$ is symmetric, nilpotent and 
    has rank $n-1$.  Then the ideal generated by the quadratic forms $\tau_{q}(q')\subset q^{\bot}$ is the ideal of the 
    unique common apolar subscheme $\Gamma$ of length $n$ of $q$ and 
    $q'$.  Moreover $\Gamma$ is a local curvilinear subscheme.
  \end{lemma} 
  \begin{proof}
  Let $\Lambda$ be a skew symmetric $n\times n$ matrix and think of 
  $A$ and $\Lambda$ as the matrices of linear endomorphisms of a 
  $n$-dimensional vector space $V$.  Then we may choose a basis
  $v_{1},\ldots,v_{n}\in V$ such that $Av_{1}=0$ and 
  $Av_{i}=v_{i-1}$ for $i=2,\ldots,n$.  Let 
  \[
  \rho:{\rm 
  Spec}(\CC[t]/t^{n})\to \PP(V): t\mapsto 
  [v_{1}+tv_{2}+\ldots+t^{n-1}v_{n}]
  \]
  and set $\Gamma=\image{\rho}$. Then $I_{\Gamma}$ is generated by 
  $\binom {n}{2}$ quadratic forms.
  We shall show that the symmetric matrices of these forms coincide 
  with the matrices $\Lambda A-A\Lambda$ as 
  $\Lambda$ varies. 
  We evaluate the quadratic form associated to $\Lambda A-A\Lambda$ on 
  the vector 
  $v=v_{1}+tv_{2}+\ldots+t^{n-1}v_{n}$:
\[
  v^{t}(\Lambda A-A\Lambda) v = v^{t}\Lambda A v- v^{t} A\Lambda v
  \]
 
  But 
    \begin{align*}
  v^{t}\Lambda A 
  v&=(v_{1}+tv_{2}+\ldots+t^{n-1}v_{n})^{t}\Lambda 
  A(v_{1}+tv_{2}+\ldots+t^{n-1}v_{n})
  \\
 & =(v_{1}+tv_{2}+\ldots+t^{n-1}v_{n})^{t}\Lambda 
  (tv_{1}+\ldots+t^{n-1}v_{n-1})
  \\
  &=(v_{1}+\ldots+t^{n-2}v_{n-1})^{t}\Lambda 
  (v_{1}+\ldots+t^{n-2}v_{n-1})t\\
  &\qquad +t^{n}v_{n}^{t}\Lambda 
  (v_{1}+\ldots+t^{n-2}v_{n-1})\\
 & =0
  \end{align*}
  since $\Lambda$ is skew symmetric and $t^{n}=0$. Therefore the quadratic forms with matrices 
  $\Lambda A-A\Lambda$ are in the ideal of $\Gamma$.  They are 
  independent and therefore generate the ideal unless $\Lambda A-A\Lambda=0$ for 
  some nontrivial $\Lambda$.  But then $\Lambda$ and $A$ commute, 
  hence have common eigenvectors.  $\Lambda$ is nontrivial and skew symmetric so it has at least $2$ 
  independent eigenvectors, while $A$ has only one, so this is 
  impossible. Clearly, $\Gamma$ is curvilinear, and any non degenerate local curvilinear subscheme of length $n$ in $\PP(V)$ is projectively equivalent to it, so the Lemma follows. \end{proof}

To complete the proof of the claim \ref{claim} and  the proof of Proposition \ref{ideal} we consider the common apolar subscheme $\Gamma=\Gamma_1\cup ...\cup \Gamma_r$ to $q$ and $q'$, and the corresponding decompositions $q=\sum q_i$ and $q'=\sum q_i'$ as above.  
By Lemma \ref{tri1}, 
\[
\sum_iU_{i}^{\bot}\cdot (U_i')^{\bot}  \subset \image\tau_q(q').
\]
 Furthermore, applying Lemma \ref{ideallocal} to each component $q_i$ and $q'_i$, the image of $\tau_{q_i}(q_i')$ in $(q^{-1}(U_i))^2$ is $I_{\Gamma_i,2}$.
But $\tau_{q_i}(q_i')$ is the restriction of $\tau_q(q')$ to $\wedge^2(q^{-1}(U_i))$, so 
\[
I_{\Gamma_i,2}\subset \image\tau_q(q') \qquad i=1,...,r
\]
and the claim and Proposition \ref{ideal} follows.
\end{proof}

 By Lemma \ref{vanish} the  quadratic forms in  $\image\tau_q(q')$ vanish in every point on any common apolar subscheme of length $n$ to $q$ and $q'$.
 Combined with Proposition \ref{ideal} it may be reasonable to guess that $\image\tau_q(q')$  is precisely the quadratic forms in the intersection of the ideals of these common apolar subschemes. We do not have a clear answer and leave this as an open question.

We are now ready to analyze our main object $VPS(Q,n)$ in its embedding in $\GG(n-1,T_{2,q})$,
i.e. as the image of the rational map
$$\gamma: \PP(T_{2,q})\dasharrow \GG(n-1,T_{2,q}).$$
We identify the restriction of the Pl\"ucker divisor to $VPS(Q,n)$.

Let $h\subset\PP(T_{1})$ be a 
hyperplane, and denote by $H_{h}\subset VSQ(Q,n)$ the set 
$$H_{h}=\{[\Gamma]\in VSQ(Q,n)|\Gamma\cap h\neq\emptyset\}.$$
\begin{lemma}\label{plucker} $H_h$ is the restriction to 
$VPS(Q,n)$ of a Pl\"ucker divisor on $\GG
({n-1},{{T_{2,q}}})$.\end{lemma}
\begin{proof}  The hyperplane $h\subset\PP(T_{1})$ is defined by some 
    $l\in S_{1}$. Let $V(l)=\{q'\in T_{2}|l(q')=0\}$, then 
    $V(l)^{\bot}=l\cdot S_{1}=\{l\cdot l'|l'\in S_{1}\}\subset S_{2}$. 
   
    For any nondegenerate subcheme $\Gamma\subset \PP(T_{1})$ of 
    length $n$, the ideal $I_{\Gamma}\subset S$ contains a reducible 
    quadric $l_{1}\cdot l_{2}$ only if $\Gamma$ intersects both 
    hyperplanes $\{l_{1}=0\}$ and $\{l_{2}=0\}$.  On the other hand the subspace of 
    quadrics $I_{\Gamma,2}\subset S_{2}$ has codimension $n$, which 
    coincides with the dimension of $l\cdot S_{1}$.  Therefore
    $$I_{\Gamma,2}\cap l\cdot S_{1}\not=\{0\}\subset S_{2} \quad{\rm if\; and \; only\; 
    if}\quad (I_{\Gamma,2})^{\bot}\cap V(l)\not=\{0\}\subset T_{2}.$$
    Notice that $\PP((I_{\Gamma,2})^{\bot})$ equals the span $\langle \Gamma\rangle\subset \PP(T_{2})$ of $\Gamma$ in the Veronese embedding. 
    
    For the lemma we now consider apolar subschemes to $q$ and the 
    projection from $\PP(T_{2})$ 
    to $\PP(T_{2,q})$.
    Since $q$ has maximal rank, $l(q)\not=0$, i.e. $q\not\in 
    V(l)$. Thus $\PP(V(l))$ is projected isomorphically to its image
    $\PP(V_{q}(l))\subset\PP(T_{2,q})$. 
    For an apolar subscheme $\Gamma$ of length $n$ the quadratic form 
    $q$ lies in the linear span of $\Gamma\subset 
    \PP(T_{2})$, so this subspace is mapped to the $(n-2)$-dimensional linear span of $\Gamma$ in $\PP(T_{2,q}).$
     We therefore deduce from the above equivalence:
    If $\Gamma$ is apolar to $q$, then the linear span
    of  $\Gamma$ in $\PP(T_{2,q})$ intersects the codimension $n$ 
    linear space $\PP(V_{q}(l))$ if and only if $\Gamma$ intersects 
    the hyperplane $h\subset \PP(T_{1})$.
    
But the set of $(n-2)$-dimensional subspaces in $\PP(T_{2,q})$ that intersect a linear space of 
codimension $n$ form a Pl\"ucker divisor, so the lemma follows.
\end{proof}

In the next section we use the special Pl\"ucker divisors $H_{h}$ of this lemma to give a local affine description of $VPS(Q,n)$, or better, the variety $VAPS(Q,n)$ of all apolar subschemes of length $n$.

\section{An open affine subvariety}\label{sec:affine}

We use a standard basis approach to compute an open affine subvariety 
of $VAPS(Q, n)$, the variety of all apolar subschemes of length $n$ to $Q$.    Of course this will include our primary object of interest, namely $VPS(Q,n)$. 
 For small $n$ there will be no difference, but for larger $n$ we have already seen that they do not coincide.
 The distinction between the two will eventually be the main concern in our analysis.  The computations in this section extensively use {\it Macaulay2} \cite {MAC2}. In particular when we show, by direct computation, that  $VAPS(Q, 6)$
is irreducible and therefore coincides with $VPS(Q,6)$ (Corollary \ref{cor:n=6}).
 
We choose coordinates such that  
$$Q=\{q=2y_{1}y_{n}+y_{2}^2+\ldots +y_{n-1}^2=0\},$$  
and consider the apolar subscheme $\Gamma_{p}$ to $q$ defined by
$$x_{1}^2, x_{1}x_{2}, x_{2}^2-x_{1}x_{n}, x_{1}x_{3}, 
x_{2}x_{3},\ldots x_{n-1}^2-x_{1}x_{n}.$$
It is of length $n$ and corresponds in the setting of the previous 
section to the intersection of the projected Veronese variety $V_{2,q}$ with the tangent space 
$T_{p}$ to $v_{2}(Q^{-1})\subset \PP(T_{2,q})$ at the point 
$v_{2}(p)=[y_{n}^2]\in \PP(T_{2,q})$ where $p=[y_{n}]=[0:\ldots:1]\in 
\PP(T_{1})$.  The tangent space to the Veronese variety $V_{2}\subset \PP(T_{2})$ at $[y_{n}^2]$ 
is spanned by 
\[
\langle y_{1}y_{n}, y_{2}y_{n},...,,y_{n-1}y_{n},y_{n}^2\rangle.
\] 
The quadric $Q^{-1}$ is defined by $\frac{1}{2} x_1x_n+\frac{1}{4}(x_2^2+...+x_{n-1}^2)$.  Its tangent space in $P(T_1)$ at $[y_n]$ is defined by $x_1$, so its tangent space in
 $P(T_{2})$ at $[y_n^2]$ is defined by $x_1$ inside the tangent space to the Veronese variety.
Therefore, the tangent space $T_p$ to  $v_{2}(Q^{-1})$ is spanned by
$$\langle y_{2}y_{n},...,y_{n-1}y_{n},y_{n}^2\rangle.$$
The orthogonal space of quadratic forms in $S_{2}$ is spanned by 
$$\langle  x_{1}^{2},x_{1}x_{2}, x_{2}^2, x_{1}x_{3}, 
x_{2}x_{3},\ldots x_{n-1}^2, x_{1}x_{n}\rangle$$ and intersect  $q^{\bot}$ 
precisely in the ideal of $\Gamma_{p}$ given above.

With reverse lexicographically order on the coordinates $x_{1},...,x_{n}$,  the 
initial ideal of $\Gamma_{p}$ is generated by the monomials
$$x_{1}^2, x_{1}x_{2}, x_{2}^2, x_{1}x_{3}, 
x_{2}x_{3},\ldots x_{n-1}^2.$$ In this monomial order, these monomials have the highest order in the 
ideal of any apolar scheme $\Gamma$ that does not intersect the hyperplane $\{x_n=0\}$.   
In fact, if the initial ideal of $\Gamma$ contains $x_ix_n$, then $x_n$ divides a quadratic form in the ideal of $\Gamma$.   But if $\Gamma$ does not intersect $\{x_n=0\}$, then $\Gamma$ would be degenerate.

We therefore consider the open subvariety $V^{\it aff}_h(n)$ containing 
$[\Gamma_{p}]$ in $VAPS(Q,n)$,  parametrizing 
apolar 
subschemes $\Gamma$ of length $n$ with support in $D(x_{n})$.  This is the complement of the 
divisor $H_{h}$ defined by $h=\{x_{n}=0\}$, the tangent hyperplane to 
$Q^{-1}$ at $[y_1]=[1:0:...:0]\in\PP(T_{1})$. 

For  $\Gamma\in V^{\it aff}_h(n)$ the initial terms of the generators of the ideal $I_{\Gamma}$ 
coincide with those of $I_{\Gamma_{p}}$. More precisely, the 
generators of $I_{\Gamma}$ may be obtained by adding suitable multiples of the 
monomials $x_{i}x_{n}, i\geq 1$ to these initial terms.   We may therefore write these generators in the form
\begin{align*}
&x_{1}^2-a_{(11,1)}x_{1}x_{n}-a_{(11,2)}x_{2}x_{n}-a_{(11,2)}x_{2}x_{n}-...-a_{(11,n)}x_{n}^2,\\
&x_{i}x_{j}-a_{(ij,1)}x_{1}x_{n}-a_{(ij,2)}x_{2}x_{n}-...-a_{(ij,n)}x_{n}^2,\quad 1\leq i< j\leq n-1,\\
&x_{i}^2-x_{1}x_{n}-a_{(ii,2)}x_{2}x_{n}-...-a_{(ii,n)}x_{n}^2,\quad 2\leq i\leq n-1.
\end{align*}
Analyzing these equations of $\Gamma$ further, we see that
the apolarity condition, i.e. that $I_{\Gamma,2}\subset q^{\bot}$, means that $a_{(11,1)}=0$ and that $a_{(ij,1)}=0$
 when $i\not= j$.  Therefore they take the form
\begin{align}\label{equa} 
   f_{11}=& x_{1}^2 -a_{({11},2)}x_{2}x_{n}  - \ldots   -a_{({11},n)}x_{n}^2,\hfill \notag   \\
   f_{12}= & x_{1}x_{2} -a_{({12},2)}x_{2}x_{n} -   \ldots   -a_{({12},n)}x_{n}^2,  \hfill    \notag  \\
 f_{22}=  &  (x_{2}^2- x_{1}x_{n}) -a_{({22},2)}x_{2}x_{n}  - \ldots    -a_{({22},n)}x_{n}^2,\hfill   \notag \\
 f_{13}=&    x_{1}x_{3} -a_{({13},2)}x_{2}x_{n} - \ldots   -a_{({13},n)}x_{n}^2, \hfill   \\
 f_{23}=&    x_{2}x_{3} -a_{({23},2)}x_{2}x_{n}   -\ldots  -a_{({23},n)}x_{n}^2 , \hfill  \notag \\
 f_{33}=&   (x_{3}^2- x_{1}x_{n}) -a_{({33},2)}x_{2}x_{n}  -\ldots   -a_{({33},n)}x_{n}^2, \hfill  \notag \\
    \vdots &    \vdots   \notag \\
  f_{(n-1)(n-1)}=&   (x_{n-1}^2- x_{1}x_{n}) -a_{({(n-1)}{(n-1)},2)}x_{2}x_{n}-\ldots  -a_{({(n-1)}{(n-1)},n)}x_{n}^2 . \notag
\end{align}

To insure that these perturbed equations 
actually define length $n$ subschemes, we ask that the first order 
relations or syzygies among the generators of $I_{\Gamma_{p}}$ lift to the entire 
family.  This is in fact precisely the requirement for the 
perturbation to define a flat family \cite[Proposition 3.1]{Art}, and will be pursued below when we find equations for $V^{\it aff}_h(n)$.

Here, we introduce weights and a torus action on this family:
We give  
\begin{itemize}
    \item  $x_{n}$ and $a_{({ij},k)}$, where $2\leq i,j,k\leq n-1$, 
weight $1$  
    \item  $x_{i}$, where $2\leq i\leq n-1$,  and $a_{({ij},n)}$, where $2\leq i,j\leq 
n-1$.  weight $2$

    \item  $x_{1}$ and $a_{({1i},n)}$ and $a_{({11},i)}$, where $2\leq i\leq 
n-1$, 
    weight $3$
 \item  $a_{({11},n)}$ weight $4$
\end{itemize}
Notice that with these weights each generator $f_{ij}$ is homogeneous. 
A $\CC^{*}$-action defined by multiplying each parameter 
with a constant $\lambda^w$ to the power of its weight, acts on
each generator by a scalar multiplication, i.e. on the total family 
in $\PP(T_{1})\times V^{\it aff}_h(n)$.
This $\CC^*$- action induces an action on the family $V^{\it aff}_h(n)$. 
In particular, if $[a]=[a_{(ij,k)}]\in V^{\it aff}_h(n)$ defines a subscheme 
$\Gamma_{[a]}$, then 
$[\lambda^{w}(a)]=[\lambda^w(a_{ij,k})]\in V^{\it aff}_h(n)$ and defines a 
subscheme $\Gamma_{[\lambda^{w}a]}$, such that $p'\in \Gamma_{[a]}$ if and 
only if $\lambda^{w}(p')\in \Gamma_{[\lambda^{w}a]}$.
 
Since $\lim_{\lambda\rightarrow 
 0}\lambda^{w}(a_{ij,k})=0$, the limit when $\lambda\to 0$ of the $\CC^*$- action is the point in $V^{\it aff}_h(n)$ representing $\Gamma_{p}$.  Thus we have shown
\begin{lemma}\label{pert}
The affine algebraic set $V^{\it aff}_h(n)$ of apolar subschemes of length $n$ contained 
in  $D(x_{n})$  coincides with the apolar schemes of length $n$ whose 
equations are 
affine perturbations of the equations of $\Gamma_{p}$.  
  
  Furthermore, the family $V^{\it aff}_h(n)$ is contractible to the point $[\Gamma_{p}]$.   
\end{lemma}

An immediate consequence is the
\begin{corollary}\label{TQ}  The apolar subscheme $\Gamma_{p}$ belongs to $VPS(Q,n)$.  In particular,
the variety of tangent spaces $TQ^{-1}\subset \GG(n-1,T_{q,2})$ to the Veronese embedding of quadric 
$Q^{-1}\subset\PP(T_{q,2})$ is a subvariety of  $VPS(Q,n)$.
\end{corollary}
Notice that $V^{\it aff}_h(n)$ depends only on $h$, and not on $p$.  Only the coordinates on $V^{\it aff}_h(n)$ depend on $p$.
On the other hand, the contractible varieties $V^{\it aff}_h(n)$ form a covering of $VAPS(Q,n)$:
\begin{lemma}\label{lem:weights}
  If  $h_j=\{l_{j}=0\}, {j=1,...,n^2}$ is a collection of tangent hyperplanes to $Q^{-1}$, 
so that no subset of $n$ of them have a common point, then the open subvarieties $V_{h_j}^{\it aff}(l_{j})$ parametrizing 
apolar 
subschemes $Z$ of length $n$ with support in $D(l_{j})$ form a covering of $VAPS(Q,n)$ of isomorphic varieties.
\end{lemma}
\begin{proof}  If an apolar subscheme $\Gamma$ has $k\leq n$ components, then the collection of hyperplanes among the $\{l_{j}=0\}$ that intersect $\Gamma$ is at most $k(n-1)<n^2$, so the $V_{h_j}^{\it aff}(l_{j})$ form a covering.
The last part follows from the homogeneity.
\end{proof}


To find equations for the family $V_h^{\it aff}(n)$ we use the parameters for the generators in (\ref{equa}), i.e.
$$a_{({ij},k)} \quad i,j\in\{1,\ldots,n-1\},\; 2\leq k \leq n$$
 where we read the first index $(ij)$ as an 
unordered pair.

It will be useful to write the generators with matrices: 
$${\small \begin{pmatrix} f_{11}\\
f_{12}\\
f_{22}\\
\vdots \\
f_{(n-1)(n-1)}\end{pmatrix}=\begin{pmatrix}1&0&...&0&-a_{(11,2)}&...&-a_{(11,n)}\\
0&1&...&0&-a_{(12,2)}&...&-a_{(12,n)}\\
0&0&...&0&-a_{(22,2)}&...&-a_{(22,n)}\\
\vdots&\vdots&&\vdots&\vdots&&\vdots\\
0&0&...&1&\vdots&...&-a_{((n-1)(n-1),n)}\end{pmatrix}\cdot \begin{pmatrix}x_{1}^2\\
x_{1}x_{2}\\
x_{2}^2-x_{1}x_{n}\\
\vdots \\
x_{n-1}^2-x_{1}x_{n}\\
x_{2}x_{n}\\
\vdots\\
x_{n}^2\end{pmatrix}.}
$$
We denote by $A_{F}$ the $\binom{n}{2}\times (\binom{n}{2}+n-1)$-dimensional coefficient matrix of these generators.
The maximal minors of $A_{F}$ are, of course, precisely the Pl\"ucker 
coordinates for $V^{\it aff}_h(n)$ in $\GG(\binom {n}{2},q^{\bot})$, or 
equivalently in $\GG(n-1,T_{2,q})$. 

We find the equations of the family by asking that the first order syzygies among the generators of $I_{\Gamma_{p}}$ lift to the entire 
family.  By \cite[Proposition 3.1]{Art}, this is precisely the requirement for the 
perturbation to define a flat family.

We use a standard basis approach (cf. \cite{Sch}).  The syzygies for a 
  subscheme $Z$ in the family are all linear, and the initial terms 
are inherited from 
  $\Gamma_{p}$.  Therefore, the difference between syzygies of $Z$ and 
  syzygies of $\Gamma_{p}$ are only multiples 
  of $x_{n}$. By the division theorem 
(\cite[Theorem A.3]{Sch}), every syzygy has the initial term 
$x_{k}(x_{i}x_{j})$, 
  where $k>j\geq i$, and has the form 
  $$x_{k}f_{ij}=\sum_{st}g_{ij}^{st}f_{st}$$
 where $f_{ij}$ is the generator with initial term $x_{i}x_{j}$  and  $g_{ij}^{st}$ is a linear form such that $g_{ij}^{st}f_{st}$ 
 has higher order than $x_{k}(x_{i}x_{j})$.   More precisely, we 
therefore consider products of the generators $(f_{ij})$ with a first order 
syzygy for 
 $\Gamma_{p}$ and 
 add precisely those multiples of $x_{n}$ in the syzygy that 
 eliminates monomials $x_{k}x_{l}x_{n}$ with $k\leq l<n$ in the 
product. 
 The relations among the parameters required for the lifting of the 
 syzygies can then be read off as the coefficients of the monomials 
 $x_{t}x_{n}^2$.

\begin{theorem}\label{thm:linear section}
The equations defining $V^{\it aff}_h(n)$ all lie in the linear span of the 
$2\times 2$ minors of the coefficient matrix $A_{F}$ of the family of 
equations 
$f_{ij}$.
In particular $VAPS(Q,n)$ is a linear section of the Grassmannian $\GG(n-1,T_{2,q})$.
\end{theorem}
\begin{proof}
    
  Consider the following first order syzygies of $\Gamma_{p}$ of rank 2 and 3:
    
    $
    R_{i}\cdot S_{i}(m), 1<i<n, m=1,2:$
    $$
    \begin{pmatrix}
        x_{1}^2 & x_{1}x_{i} & x_{i}^2-x_{1}x_{n}
    \end{pmatrix}\cdot 
    \begin{pmatrix}

        -x_{i} & x_{n}  \\
        x_{1} & -x_{i}  \\
        0 & x_{1}
    \end{pmatrix}=0,
     $$
     where $S_{i}(m)$ is the $m$-th column vector in the syzygy matrix $S_i$, 
     
     $R_{ij}\cdot S_{ij}(m), 1<i<j<n, m=1,\ldots, 4:$
     $$
     \begin{pmatrix}
         x_{i}^2-x_{1}x_{n} & x_{1}x_{i} & x_{1}x_{j} & x_{i}x_{j} & 
x_{j}^2-x_{1}x_{n}
    \end{pmatrix}\cdot 
    \begin{pmatrix}
        x_{j} & 0 &0 &0  \\
        0 & x_{n} &x_{j} &0 \\
x_{n} & 0&-x_{i} & -x_{i} \\
-x_{i} & -x_{j}&0 & x_{1} \\
        0 & x_{i}&0 &0
    \end{pmatrix}=0,
    $$
    and 
    
    $
    R_{ijk}\cdot S_{ijk}(m), 1<i<j<k<n, m=1,2:$
    $$
     \begin{pmatrix}
          x_{i}x_{j} & x_{j}x_{k} & x_{i}x_{k} 
    \end{pmatrix}\cdot 
    \begin{pmatrix}
        -x_{k} & -x_{k}  \\
        x_{i} & 0  \\
0 & x_{j} \\
\end{pmatrix}=0.
    $$
These syzygies are clearly linearly independent, and their number 
$2\binom{n}{3}$
coincides with the dimension of the space of first 
order syzygies, according to Proposition \ref{betti numbers}, so they form a basis.

We lift these syzygies by adding the multiples 
of $x_{n}$ in the syzygy matrix, that reduces the product to cubic 
polynomials with monomials only of the form
$x_{i}x_{n}^2$.   We denote by $\tilde S_{i}(j)$ the syzygies 
obtained from $S_{i}(j)$ this way.  
Likewise we denote by $\tilde R_{i}$ the row vector obtained from $R_{i}$ by
substituting the entries $x_{s}x_{t}$ by $f_{st}$.  Similarly we get 
row vectors $\tilde R_{ij}, \tilde R_{ijk}$ and column vectors $\tilde 
S_{ij}(r)$ and $\tilde S_{ijk}(r)$.

For example, with $i=2$ and $n=4$ we get
$$\begin{array}{cc}\tilde R_{2}\cdot 
\tilde 
S_{2}(1)=&\begin{pmatrix}f_{11}& f_{12} & f_{22} & f_{13} & f_{23} & 
f_{33}\\ \end{pmatrix}\cdot\begin{pmatrix} 
-x_{2}\\ 
x_{1}+a_{(12,2)}x_{4}\\-a_{(11,2)}x_{4}\\a_{(12,3)}x_{4}\\-a_{(11,3)}x_{4}\\0\\
\end{pmatrix}\hfill\\
\hfill =&(-a_{(12,4)}+a_{(11,2)})x_{1}x_{4}^2\hfill\\
&+(a_{(11,4)}-a_{(12,2)}^{2}+a_{(22,2)}a_{(11,2)}-a_{(13,2)}a_{(12,3)}
+a_{(23,2)}a_{(11,3)})x_{2}x_{4}^2\hfill\\
&+(-a_{(12,3)}a_{(12,2)}+a_{(22,3)}a_{(11,2)}-a_{(13,3)}a_{(12,3)}
+a_{(23,3)}a_{(11,3)})x_{3}x_{4}^2 \hfill\\
&+(-a_{(12,4)}a_{(12,2)}+a_{(22,4)}a_{(11,2)}-a_{(13,4)}a_{(12,3)}
+a_{(23,4)}a_{(11,3)})x_{4}^3.\hfill\\
\end{array}$$
       
  For general $i$ and $n$ we get (with the first pair in the index 
  unordered to simplify presentation of the summation)
         $$\begin{array}{cc}\tilde R_{i}\cdot\tilde 
	 S_{i}(1)=&(a_{({11},i)}-a_{({1i},n)})x_{1}x_{n}^2+a_{(11,n)}x_{i}x_{n}^2\hfill\\
         
        & 
	+\sum_{j=2}^n(\sum_{k=2}^{n-1}(a_{(11,k)}a_{(ik,j)}-a_{(1i,k)}a_{(1k,j)})x_{j}x_{n}^2.\hfill\\
	\end{array}$$

Similarly 
$$\begin{array}{cc}
\tilde R_{i}\cdot \tilde S_{i}(2)=&(a_{({1i},i)}-a_{({ii},n)})x_{1}x_{n}^2+a_{(1i,n)}x_{i}x_{n}^2\hfill\\
&+\sum_{j=2}^n(-a_{({11},j)}+\sum_{k=2}^{n-1}(a_{(1i,k)}a_{(ik,j)}-a_{(ii,k)}a_{(1k,j)}))x_{j}x_{n}^2,\hfill\\

\tilde R_{ij}\cdot \tilde S_{ij}(1)=&(a_{({ij},i)}-a_{({ii},j)})x_{1}x_{n}^2+a_{(ij,n)}x_{i}x_{n}^2-a_{(ii,n)}x_{j}x_{n}^2\hfill\\

&+\sum_{k=2}^n(-a_{(1j,k)}+\sum_{m=2}^{n-1}(a_{(ij,m)}a_{(im,k)}-a_{(ii,m)}a_{(jm,k)}))x_{k}x_{n}^2,\hfill\\

 \tilde R_{ij}\cdot \tilde S_{ij}(2)=&(a_{({ij},j)}-a_{({jj},i)})x_{1}x_{n}^2-a_{(jj,n)}x_{i}x_{n}^2+a_{(ij,n)}x_{j}x_{n}^2\hfill\\
 
 &+\sum_{k=2}^n(-a_{(1i,k)}+\sum_{m=2}^{n-1}(a_{(ij,m)}a_{(jm,k)}-a_{(jj,m)}a_{(im,k)}))x_{k}x_{n}^2,\hfill\\
 
 \tilde R_{ij}\cdot \tilde S_{ij}(3)=&(a_{({1j},i)}-a_{({1i},j)})x_{1}x_{n}^2+a_{(1j,n)}x_{i}x_{n}^2-a_{(1i,n)}x_{j}x_{n}^2\hfill\\
&+\sum_{k=2}^n(\sum_{m=2}^{n-1}(a_{(1j,m)}a_{(im,k)}-a_{(1i,m)}a_{(jm,k)}))x_{k}x_{n}^2,\hfill\\

\tilde R_{ij}\cdot \tilde S_{ij}(4)=&(a_{({1j},i)}-a_{({ij},n)})x_{1}x_{n}^2-a_{(1j,n)}x_{i}x_{n}^2\hfill\\

&+\sum_{k=2}^n(\sum_{m=2}^{n-1}(a_{(1j,m)}a_{(im,k)}-a_{(ij,m)}a_{(1m,k)}))x_{k}x_{n}^2,\hfill\\

\tilde R_{ijk}\cdot \tilde S_{ijk}(1)=&(a_{({ij},k)}-a_{({jk},i)})x_{1}x_{n}^2-a_{({jk},n)}x_{i}x_{n}^2+a_{({ij},n)}x_{k}x_{n}^2\hfill\\

&+\sum_{l=2}^n(\sum_{m=2}^{n-1}(a_{(ij,m)}a_{(km,l)}-a_{(jk,m)}a_{(im,l)}))x_{l}x_{n}^2,\hfill\\

\tilde R_{ijk}\cdot \tilde S_{ijk}(2)=&(a_{({ij},k)}-a_{({ik},j)})x_{1}x_{n}^2-a_{({ik},n)}x_{j}x_{n}^2+a_{({ij},n)}x_{k}x_{n}^2\hfill\\

&+\sum_{l=2}^n(\sum_{m=2}^{n-1}(a_{(ij,m)}a_{(km,l)}-a_{(ik,m)}a_{(jm,l)}))x_{l}x_{n}^2.\hfill\\
\end{array}$$
The linear relations in the parameters of the family $V^{\it aff}_h(n)$ are precisely the coefficients of $x_{1}x_{n}^2$ in these products:

\begin{lemma}\label{lem:linear}
 The space of linear forms in the ideal of $V^{\it aff}_h(n)$ is generated by the 
 following
 forms, where $ \{i,j,k\}$ is 
any subset of distinct elements in $\{2,\ldots,n-1\}$
 $$ a_{({11},i)}-a_{({1i},n)},\; a_{({1i},i)}-a_{({ii},n)},\;
 a_{({1j},i)}-a_{({ij},n)},\; a_{({ij},j)}-a_{({jj},i)},  \;
  a_{({ij},k)}-a_{({jk},i)}.
  $$
\qed \end{lemma}
Notice that only the first two occur when $n=3$, and only the first 
four occur when $n=4$. 

Using these linear relations, the quadratic ones all become linear in 
the $2\times 2$ minors of the matrix $A_{F}$ of coefficients $a_{(ij,k)}$, i.e. linear in the Pl\"ucker coordinates.
  In fact, by a straightforward but tedious derivation from the above presentation, we 
  may write the generators of the ideal $V^{\it aff}_h(n)$ as linear combinations 
  of $2\times 2$ minors in the coefficient matrix $A_{F}$  (cf. 
  the documented computer algebra code to perform the computation of 
ideal  generators  \cite{RSMAC}):
  
  \begin{lemma}\label{lem:V3}
 Modulo the linear forms the ideal of $V^{\it aff}_h(3)$ is generated by 
 $$ a_{(11,3)}+a_{(12,3)}a_{(22,2)}-a_{(12,2)}a_{(22,3)}.$$
	\qed\end{lemma}

   \begin{lemma}\label{lem:V4}
 Modulo the linear forms the ideal of $V^{\it aff}_h(4)$ is generated by 
 $$-a_{(12,2)}-a_{(13,3)}+(a_{(23,2)}a_{(22,3)}-a_{(22,2)}a_{(23,3)})+(a_{(33,2)}a_{(23,3)}-a_{(23,2)}a_{(33,3)}),
$$
$$-a_{({11},2)}+(a_{(12,3)}a_{(23,2)}-a_{(23,3)}a_{(12,2)})+(a_{(13,3)}a_{(33,2)}-a_{(33,3)}a_{(13,2)}),$$
$$-a_{({11},3)}+(a_{(12,2)}a_{(22,3)}-a_{(22,2)}a_{(12,3)})+(a_{(13,2)}a_{(23,3)}-a_{(23,2)}a_{(13,3)}),$$
  $$(a_{(11,2)}a_{(22,3)}-a_{(12,2)}a_{(12,3)})+(a_{(11,3)}a_{(23,3)}-a_{(12,3)}a_{(13,3)}),$$
$$a_{(11,4)}+(a_{(12,4)}a_{(22,2)}-a_{(22,4)}a_{(12,2)})+(a_{(13,4)}a_{(23,2)}-a_{(23,4)}a_{(13,2)}),$$
  $$a_{(11,4)}+(a_{(12,4)}a_{(23,3)}-a_{(23,4)}a_{(12,3)})+(a_{(13,4)}a_{(33,3)}-a_{(33,4)}a_{(13,3)}).$$
\qed\end{lemma}

\begin{lemma}\label{lem:Vn}
 Modulo the linear forms the ideal of $V^{\it aff}_h(n)$,  is 
 generated by the following forms:
 
 For $i\in\{2,\ldots,n-1\},$
 $$a_{(11,n)}- 
 \sum_{m=2}^{n-1}(a_{(im,n)}a_{(1m,i)}-a_{(1m,n)}a_{(im,i)}),$$
  for any subset $\{i,j\}\subset \{2,\ldots,n-1\},$
  $$ a_{(11,i)}- 
\sum_{m=2}^{n-1}(a_{(jm,n)}a_{(im,j)}-a_{(im,n)}a_{(jm,j)}),$$
  $$\sum_{m=2}^{n-1}(a_{(1m,n)}a_{(im,j)}-a_{(im,n)}a_{(1m,j)})$$ and 
 $$ a_{(1i,i)}+a_{(1j,j)}-\sum_{m=2}^{n-1}(a_{(jm,i)}a_{(im,j)}-a_{(im,i)}a_{(jm,j)}),$$
   for any subset $\{i,j,k\}\subset \{2,\ldots,n-1\},$
$$a_{(1j,k)}-\sum_{m=2}^{n-1}(a_{(jm,i)}a_{(im,k)}-a_{(im,i)}a_{(jm,k)})$$ and $$\sum_{m=2}^{n-1}(a_{(jm,n)}a_{(im,k)}-a_{(im,n)}a_{(jm,k)}),$$
and for any subset $\{i,j,k,l\}\subset \{2,\ldots,n-1\},$
$$\sum_{m=2}^{n-1}(a_{(im,j)}a_{(km,l)}-a_{(km,j)}a_{(im,l)}).$$
 \qed\end{lemma}
  
Since the open affine sets $V^{\it aff}_h(n)$ cover $VAPS(Q,n)$, we conclude that 
$VAPS(Q,n)$ is a linear section of the Grassmannian $\Gr{\binom{n}{2}}{q^{\bot}}$.   Equivalently, $VAPS(Q,n)$ is 
projectively equivalent to a linear section of $\Gr{n-1}{T_{2,q}}$ in 
its Pl\"ucker embedding. This concludes the proof of Theorem 
\ref{thm:linear section}.
\end{proof}

Using the linear relations we may reduce the number of variables when 
$n>4$, and use as indices the following unordered three element sets:
$${\mathcal I}=\{\{11k\}|1<k\leq n\}\cup \{\{1jk\}|1<j\leq k<n\}\cup 
\{\{ijk\}|1< i\leq j\leq k<n \}.$$
Let $R=\CC[a_{I}|I\in\mathcal I]$. 
We substitute $a_{11k}=a_{11,k}, a_{1jk}=a_{1j,k}, a_{ijk}=a_{ij,k}$ 
and get the following generators for the ideal of $V^{\it aff}_h(n)$ in $R$, where $ \{i,j,k,l\}$ is 
any subset of $\{2,\ldots,n-1\} $:

\begin{lemma}\label{symgens} The ideal of $V^{\it aff}_h(n)$ is generated by the 
    following polynomials in $R$:
For $i\in\{2,\ldots,n-1\}$,
 $$a_{11n}-\sum_{m=2}^{n-1}(a_{1im}^{2}-a_{11m}a_{iim}),$$
  for any subset $\{i,j\}\subset \{2,\ldots,n-1\}$,
$$ a_{11i}- \sum_{m=2}^{n-1}(a_{1jm}a_{ijm}-a_{1im}a_{jjm}),$$
  $$\sum_{m=2}^{n-1}(a_{11m}a_{ijm}-a_{1im}a_{1jm}),\quad 
  a_{1ii}+a_{1jj}-\sum_{m=2}^{n-1}(a_{ijm}^{2}-a_{iim}a_{jjm}),$$
   for any subset $\{i,j,k\}\subset \{2,\ldots,n-1\}$,
$$a_{1jk}-\sum_{m=2}^{n-1}(a_{ijm}a_{ikm}-a_{iim}a_{jkm}),\quad 
\sum_{m=2}^{n-1}(a_{1jm}a_{ikm}-a_{1im}a_{jkm}),$$
and for any subset $\{i,j,k,l\}\subset \{2,\ldots,n-1\}$,
 $$\sum_{m=2}^{n-1}(a_{ijm}a_{klm}-a_{jkm}a_{ilm}).$$
  \qed \end{lemma}

Notice that these generators are all homogeneous in the weights introduced above.
 
The linear parts of the ideal generators define the tangent space of the family $V^{\it aff}_h(n)$ at $[\Gamma_p]$, so another consequence of our 
computations is the tangent space dimension.  

\begin{proposition}\label{prop:Ln} Let $L(n)$ be the space of 
    linear forms spanned by the linear parts of the generators in the 
    ideal of $V^{\it aff}_h(n)$.  Then $L(3)$ is spanned by
$$ a_{({11},3)}, \quad a_{({11},2)}-a_{({12},3)}\quad {\it and}\quad  a_{({12},2)}-a_{({22},3)}.$$
 $L(4)$ is spanned by 
 $$ a_{({11},4)},\quad a_{({11},2)},\quad a_{({12},4)},\quad 
a_{({11},3)},\quad a_{({13},4)}, $$ $$
   a_{({12},2)}-a_{({22},4)},\quad a_{({13},3)}-a_{({33},4)},  
\quad a_{({12},2)}+a_{({13},3)},$$
  $$  a_{({12},3)}-a_{({23},4)},\quad a_{({13},2)}-a_{({12},3)}, \quad
   a_{({23},3)}-a_{({33},2)}\quad {\it and}\quad a_{({22},3)}-a_{({23},2)}.$$
   $L(n)$, when 
  $n>4$, is spanned by $a_{(11,n)}$ and for any $i\in \{2,\ldots,n-1\},$
 $$ a_{({11},i)},a_{({1i},n)}, a_{({1i},i)},a_{({ii},n)},$$
 for any subset $\{i,j\}\subset \{2,\ldots,n-1\},$
  $$a_{({1j},i)},a_{({ij},n)}, a_{({ii},j)}-a_{({ij},i)},$$
  and for any subset $\{i,j,k\}\subset \{2,\ldots,n-1\},$
$$ a_{({ij},k)}-a_{({jk},i)}.$$
 
  In particular $$V^{\it aff}_h(3)\cong\AAA^{3},\;V^{\it 
  aff}_p(4)\cong\AAA^{6},\; V^{\it aff}_h(5)\cong\AAA^{10}.$$
\end{proposition}
\begin{proof}
The linear parts of the generators can be read off Lemma \ref{lem:linear} 
and Lemmas \ref{lem:V3}, \ref{lem:V4} and \ref{lem:Vn}.  Notice only that the two term 
forms $$a_{({12},2)}-a_{({22},4)},\quad a_{({13},3)}-a_{({33},4)},  
\quad a_{({12},2)}+a_{({13},3)}$$ span a three dimensional space, 
while $$a_{({1i},i)}-a_{({ii},n)},\quad a_{({1j},j)}-a_{({jj},n)},  
\quad a_{({1i},i)}+a_{({1j},j)},\quad {\rm for}\;\; 1<i<j<n$$ span the 
space generated by 
$$a_{({1i},i)},\quad a_{({ii},n)},\quad  1<i<n$$ when $n>4$.

For $n=3$, the family $V^{\it aff}_h(3)$ has $6$ parameters, while there are 
three independent linear forms in the relations 
so the tangent space at $[\Gamma_{p}]$ has dimension $6-3=3$ as expected.
In fact $V^{\it aff}_h(3)\cong\AAA^{3}$ with parameters $a_{(11,2)}, 
a_{(12,2)},a_{(22,2)}$. 

For $n=4$ the family $V^{\it aff}_h(4)$ has $18$ parameters, while the linear 
forms in the relations are generated by $12$ independent forms,
so the tangent space at $[\Gamma_{p}]$ has dimension $18-12=6$. In fact
$V^{\it aff}_h(4)\cong\AAA^{6}$ with parameters $a_{(12,2)}, 
a_{(12,3)},a_{(23,3)},a_{(22,3)},a_{(22,2)},a_{(33,3)}$.

For $n>4$ we see that all parameters with a $1$ or an $n$ in the 
index are independent forms in the space of linear parts of ideal 
generators in $V^{\it aff}_h(n)$.  Furthermore, the other linear parts, simply 
expresses that $\{(ijk)|1<i\leq j\leq k<n\}$ form a natural index set 
for representatives of the parameters.
The cardinality  of this index set is simply the cardinality of monomials of 
degree $3$ in $n-2$ variables, i.e. $\binom{n}{3}.$
In case $n=5$ we again conclude that $V^{\it aff}_h(5)\cong\AAA^{10}$ with 
parameters $\{a_{ijk}|2\leq i\leq j\leq k\leq 4\}$.
\end{proof}

\begin{corollary}\label{cor:tangentdimension}
  The tangent space dimension of $VAPS(Q,n)$ at $[\Gamma_{p}]$ is 
  $\binom{n}{3}$ when $n> 5$.
  When $n\leq 5$, $VAPS(Q,n)$ has a finite cover of affine spaces, in 
  particular $VAPS(Q,n)$ is smooth and coincides with $VPS(Q,n)$.
  \qed\end{corollary}
  \begin{remark}\label{haslines}  Let $\Gamma$ be a smooth apolar subscheme to $Q$ consisting of $n$ distinct points. Any subset of $n-2$ points in $\Gamma$ is contained
   in a pencil of apolar subschemes that form a line in $VPS(Q,n)$ through $[\Gamma]$.
  Thus $\binom{n}{2}$ lines in $VPS(Q,n)$ through $[\Gamma]$ is contained in the tangent space at $[\Gamma]$.
  \end{remark}

  We extend this remark and give a conceptual reason for the 
  dimension of the tangent space to $VAPS(Q,n)$ at $[\Gamma_{p}]$.

\begin{proposition}\label{sing2}
Let $[\Gamma_{p}]\in VPS(Q,n)\subset\GG(n-1,T_{2,q})$ be a point on  the subvariety $TQ^{-1}$ 
in its Grassmannian embedding.  
Then $VPS(Q,n)$  
contains the cone over a $3$-uple embedding of $\PP^{n-3}$ with 
vertex at $[\Gamma_{p}]$.
\end{proposition}
\begin{proof}
    We first identify a cone over  a $3$-uple embedding of $\PP^{n-3}$ inside $VAPS(Q,n)$, and then give an explicit description of the apolar subschemes parameterized by this cone in order to show that the cone is contained in $VPS(Q,n)$.
  
  Consider the subvariety $V^{\it vero}_{p}(n)\subset V^{\it aff}_h(n)$ parameterizing ideals 
  $I_{\Gamma}$ with coefficient matrix $A_{F}(\Gamma)=(I\;A)$ where 
  the submatrix
$A=(a_{(ij,k)})$ has rank at most $1$ and has nonzero entries only in the submatrix 
$A_{0}\subset A$ with entries $\{a_{(ij,k)}| 1<i\leq j<n, 1<k<n\}$.
As above, using the linear relations, we may substitute the parameters 
$a_{(ij,k)}$ with parameters $a_{ijk}$ whose indices are unordered 
triples $(ijk)$. In these new parameters the matrix $A_{0}$ takes the form:
$$\begin{pmatrix}
a_{222}&a_{223}&...&a_{22(n-1)}\\
a_{223}&a_{233}&...&a_{23(n-1)}\\
...&...&...&...\\
a_{22(n-1)}&a_{23(n-1)}&...&a_{2(n-1)(n-1)}\\
...&...&...&...\\
a_{2(n-1)(n-1)}&a_{3(n-1)(n-1)}&...&a_{(n-1)(n-1)(n-1)}\\
\end{pmatrix}
$$
By Theorem \ref{thm:linear section} the equations of $V^{\it aff}_h(n)$ are 
linear in the $2\times 2$ minors of the coefficient matrix $A$, so 
any rank $1$ matrix $A_{0}$ defines a point on $V^{\it vero}_{p}(n)$.  
The symmetry in the indices explains why the $2\times 2$ minors 
of the matrix define the $3$-uple embedding of $\PP^{n-3}$.  Since 
the ideal of
$\Gamma_{p}$ correspond to the zero matrix, we conclude that the subvariety $V^{\it vero}_{p}(n)$ in $VAPS(Q,n)$
is the cone over this $3$-uple embedding.  

To see that 
$V^{\it vero}_{p}(n)$ is contained in $VPS(Q,n)$ we show that a general point on $V^{\it vero}_{p}(n)$ lies in the closure of smooth apolar subscheme to $Q$.  
For this, we describe for each general point $s\in \PP^{n-3}$ an apolar 
subscheme $\Gamma_s$ belonging to $V^{\it vero}_{p}(n)$.  It  has two components $\Gamma_s=\Gamma_{s,0}\cup p_s$, the first one $\Gamma_{s,0}$ of length $n-1$ and supported at $p$, while the second component $p_s$ is a closed point.
We shall show that $q$ has a decomposition $q=q_l+q(l)^2\in T_2$ where $[q(l)]=p_s\in \PP(T_1)$ and $q_l\in (l^{\bot})^2$.
 The subscheme $\Gamma_{s,0}$ is apolar to $q_l$ and contains the first order neighborhood of $p$ inside the quadric $\{q_l^{-1}=0\}\subset \PP(l^{\bot})$ in the hyperplane polar to $p_s$.
Then $\Gamma_{s,0}$ lies in the closure of smooth apolar subschemes to $q_l$.  We conclude by applying  Proposition 
\ref{ort}.
 
Let 
$s=[s_{2}:..:s_{n-1}]\in \PP^{n-3}$ and let 
\[
||s||^2=s_{2}^{2}+\ldots+s_{n-1}^{2},\quad \langle s,x\rangle=\sum_{i=2}^{n-1}s_ix_i,\quad \langle s,y\rangle=\sum_{i=2}^{n-1}s_iy_i, 
\] then 
\[
x_{1}^{2},\ldots,x_{1}x_{n-1},\]
\[
x_{i}^{2}-x_{1}x_{n}-s_{i}^{2}\langle s,x\rangle x_{n}\quad 
1<i<n,\quad
x_{i}x_{j}-s_{i}s_{j}\langle s,x\rangle x_{n}\quad 
1<i<j<n
\]
defines a subscheme $\Gamma_{s}$ that belongs to $V^{\it vero}_{p}(n)$.  When $||s||^2\not=0$, then 
$\Gamma_{s}$ contains the point 
$$p_{s}=[q(||s||^2\langle s,x\rangle+x_1)]=[||s||^2\langle s,y\rangle+y_n]\in \PP(T_1)$$
Consider the linear subspace 
$L_{s}=\{x_{1}=0\}\cap\{\langle s,x\rangle=0\}$.  The 
intersection $\Gamma_{s}\cap L_{s}$ is the 
subscheme defined by $$x_{2}^{2}=x_{2}x_{3}=\ldots =x_{n-1}^{2}=0.$$  
This subscheme has length $n-2$.  The union $p_{s}\cup(\Gamma_{s}\cap 
L_{s})$  
spans the hyperplane $\{x_{1}=0\}$, so the residual point in 
$\Gamma_{s}$ is the pole, with respect to $Q^{-1}$, of this 
hyperplane, i.e. the point $p$.  
Therefore the subscheme $\Gamma_{s,0}=\Gamma_{s}\setminus p_{s}$ has length $n-1$, is supported 
in $p$, and contains the first order neighborhood of $p$ in the 
codimension two linear space $L_{s}$.

The subscheme $\Gamma_{s,0}$ is apolar to the quadric
$$q_{s}=(||s||^2\langle s,y\rangle+y_n)^2-||s||^6(2y_{1}y_{n}+y_2^2+...+y_{n-1}^2).$$

Let $l=||s||^2\langle s,x\rangle+x_1$.  Then $p_s=[q(l)]$, while $$l^{\bot}=\langle y_2-||s||^2s_iy_1,...,y_{n-1}-||s||^2s_iy_1, y_n\rangle.$$
Then $q_s\in(l^{\bot})^2<2$ and
$$(q(l))^2-q_s=||s||^6\cdot q\in T_2$$
According to Proposition 
\ref{ort} a subscheme $\Gamma_{0}$ in $\PP(l^{\bot})$ of length $n-1$ is apolar to $q_s$ if and only if $\Gamma=\Gamma_{0}\cup p_s$ is apolar to $q$.   Now, $\Gamma_{s,0}$ is apolar to $q_s$ and contains a first order neighborhood of a point on the smooth quadric $\{{q_s}^{-1}=0\}$ in $\PP(l^{\bot})\subset \PP(T_1)$.  
By Remark \ref{gammap}, the subscheme $\Gamma_{s,0}$ is a subscheme like $\Gamma_p$, with respect to $q_s$. Therefore $\Gamma_{s,0}$  lies in the closure of smooth apolar subschemes to $q_s$.
But then $\Gamma_s$ must lie in the closure of smooth apolar subschemes to $q$.  Hence $[\Gamma_{s}]\in VPS(q,n)$.
 
\end{proof}

\begin{corollary}\label{sing}  $VPS(Q,n)$ is singular when $n\geq 6$. 
\end{corollary}
\begin{proof} The cone with vertex at $[\Gamma_{p}]\in TQ^{-1}$ over the $3$-uple embedding of $\PP^{n-3}$ is 
    contained in the tangent space of $VPS(Q,n)$ at $[\Gamma_{p}]$, 
    i.e. also in the tangent space of $VAPS(Q,n)$.  Since the span of 
    the cone and the tangent space of the latter have the same 
    dimension, they coincide.  In particular the tangent space of 
    $VPS(Q,n)$ at $[\Gamma_{p}]$ has dimension $\binom{n}{3}$.  When 
    $n\geq 6$, then $\binom{n}{3}>\binom{n}{2}={\rm dim}VPS(Q,n)$ so 
    $VPS(Q,n)$ is singular.
    \end{proof}

We pursue the case $n=6$ a bit further and show that $VAPS(Q,6)$ and $VPS(Q,6)$ coincide.
We use the symmetric variables  
$$a_{ijk}=a_{(ij,k)}, \; 1\leq i,j,k\leq 6$$
for any permutation of the letters $i,j,k$.  
According to Lemma \ref{symgens} we may list the generators 
explicitly.  This list is however not minimal.  In fact, a minimal 
set of generators is given by the following twenty generators in 
weight $2$,  four generators in weight $3$ and one generator in weight $4$.
  The twenty generators of weight $2$ are the generators of weight 
  $2$ in Lemma \ref{symgens}:
    For each $1<k<6$, and each pair 
    $\{i,j\}\subset\{2,3,4,5\}\setminus\{k\}$ the generator 
 $$ -a_{ij6}+\sum_{m=2}^{5}(a_{ikm}a_{jkm}-a_{ijm}a_{kkm}),$$
  for each pair  $\{i,j\}\subset\{2,3,4,5\}$ the generator
$$ -a_{ii6}-a_{jj6}+\sum_{m=2}^{5}(a_{ijm}a_{ijm}-a_{iim}a_{jjm}),$$
and additionally the two generators
$$\sum_{m=2}^{5}(a_{23m}a_{45m}-a_{24m}a_{35m})\quad{\rm and}\quad\sum_{m=2}^{5}(a_{23m}a_{45m}-a_{25m}a_{34m}).$$

The last five generators are computed from the list of Lemma 
\ref{symgens} using {\it Macaulay2} \cite {MAC2}, see the documented code in \cite{RSMAC}.

Of weight $3$ we find, for $i=2,3,4$:
$$a_{11i}-\sum_{m=2}^{5}(a_{im6}a_{m55}-a_{m56}a_{im5})$$ and
$$a_{115}-\sum_{m=2}^{5}(a_{m46}a_{m45}-a_{m56}a_{m44}).$$

The generator of weight $4$ is
$$a_{116}-\sum_{m=2}^5 a_{m56}^2+\sum_{m=2}^5a_{11m}a_{m55}.$$

The ten parameters with $6$ in the index appear linearly in the  $20$ 
generators of weight $2$, while the five parameters with $11$ in the 
index appear linearly in the five generators of weights $3$ and $4$.  
The remaining $10$ generators of weight $2$ therefore depend only on 
$20$ parameters $a_{I}$. In fact they depend only on $16$ linear 
forms.  It is a remarkable fact that these ten quadratic forms define the $10$-dimensional spinor variety.  To see this we choose and rename the following $16$ forms:
$$ x_{1234}=-a_{353}+a_{252},\; x_{15}=-a_{555}+a_{454}+a_{353}+a_{252},\; x_{34}=a_{453},
$$
$$x_{1235}=a_{554}-a_{444}+a_{343}+a_{242},\; 
x_{14}=a_{343}-a_{242},\;x_{35}=a_{553}-a_{232}, 
$$
$$x_{1245}=-a_{553}-a_{443}+a_{333}-a_{232},\; 
x_{13}=a_{553}-a_{443},\; x_{24}=a_{452},$$
$$x_{1345}=a_{552}+a_{442}+a_{332}-a_{222},\; 
x_{12}=a_{552}-a_{442},\; x_{23}=a_{352},
$$ 
$$x_{2345}=a_{454}-a_{353},\; x_{45}=a_{554}-a_{242},\; 
x_{25}=-a_{442}+a_{332},\; 
x_{0}=a_{342}.$$
In these variables the ten quadratic generators takes the form
$$q_{0}=   x_{25}x_{34} - x_{35}x_{24} + x_{45}x_{23} + x_{2345}x_{0}, $$
$$
q_{1}=-x_{45}x_{13} + x_{14}x_{35} - x_{15}x_{34} + x_{1345}x_{0},$$
$$
q_{2}=x_{45}x_{12}+     x_{14}x_{25} + x_{15}x_{24} +x_{1245}x_{0}, $$
$$
q_{3}=-x_{35}x_{12}+    x_{13}x_{25} - x_{15}x_{23} + x_{1235}x_{0}, $$
$$
q_{4}=x_{12}x_{34} - x_{13}x_{24} + x_{14}x_{23} + x_{1234}x_{0}, $$
$$
q_{5}= x_{1345}x_{12} + x_{1245}x_{13}+ x_{1235}x_{14}+x_{15}x_{1234},$$
$$
q_{6}=-x_{2345}x_{12} + x_{1245}x_{23} + x_{1235}x_{24}+x_{1234}x_{25},  $$
$$
q_{7}=-x_{2345}x_{13}- x_{1345}x_{23} + x_{1235}x_{34} +x_{1234}x_{35} , $$
$$
q_{8} =- x_{2345}x_{14} - x_{1345}x_{24}- x_{1245}x_{34} +x_{1234}x_{45},$$
$$
q_{9}=-x_{15}x_{2345} - x_{1345}x_{25} - x_{1245}x_{35} - x_{1235}x_{45}.
$$
The first five express (when $x_0=1$) the variables $x_{ijkl}$ as quadratic Pfaffians in the $x_{st}$, while the last five quadrics express the linear syzygies among these Pfaffians.
The ten quadratic forms satisfy the following quadratic relation
$$
q_{0}q_{5} + q_{1}q_{6} + q_{2}q_{7} + q_{3}q_{8} + q_{4}q_{9}=0.
$$
In fact the ten quadratic forms  generate the ideal of the $10$-dimensional spinor variety embedded in $\PP^{15}$ by its spinor 
coordinates \cite[Section 6]{RS},\cite{Mu95}.

\begin{corollary}\label{cor:n=6}  $V^{\it aff}_h(6)$ is isomorphic to a cone over the 
ten-dimensional spinor variety embedded in $\PP^{15}$ by its spinor 
coordinates.  In particular $VAPS(Q,6)$  is singular, irreducible and coincides with $VPS(Q,6)$. \qed
\end{corollary}


We end this section summarizing some computational 
results, for small $n$, 
of some natural subschemes of $VAPS(Q,n)$.  
The first is the punctual part $ V^{\it loc}_p(n)$ of $VAPS(Q,n)$, i.e. the variety of apolar subschemes in 
$VAPS(Q,n)$ with support at a single point $p$.  
The support $p$ of a local apolar subscheme must lie on $Q^{-1}$ by Lemma \ref{nonred}.
Therefore we may assume that $p=[0:0:...:1]$, and use the equations \ref{equa}.
Of course, $[\Gamma_{p}]$ is then in $V^{\it loc}_p(n)$.  Furthermore,  
$V^{\it loc}_p(n)$ is 
naturally contained in a second natural subscheme of $VAPS(Q,n)$, namely  $V^{\it sec}_{p}(n)$, the variety of apolar subschemes in 
$V^{\it aff}_h(n)$ that contains the point $p$.  

 We will do the explicit computation in the cases where $VAPS(Q,n)=VPS(Q,n)$ is smooth, i.e. when $n<6$.
An apolar subscheme in $V^{\it aff}_h(n)$ lies in $V^{\it sec}_{p}(n)$ if and only if 
the term $x_{n}^2$ does not appear in any equation, so $V^{\it sec}_{p}(n)$ is 
defined by the equations $a_{(ij,n)}=0$ for $1\leq i\leq j<n$ in $V^{\it aff}_h(n)$.
 The linear relations then imply that $a_{(1i,j)}=0$ for all $i$ and 
 $j$, and as before that each parameter $a_{(ij,k)}$ with $1<i,j,k<n$ 
 may be represented by a parameter 
 $$a_{ijk}\; {\rm with}\; 1<i\leq j\leq k<n.$$
 
 In particular for $n=3$ the only parameter left is $a_{222}$, and 
 $V^{\it sec}_p(3)$ is isomorphic to the affine line.  The equations $x_{1}^2=x_{1}x_{2}=x_{2}^2-x_{1}x_{3}-a_{222}x_{2}x_{3}$
 define a scheme supported at $p$ only if $a_{222}=0$, so $V^{\it loc}_p(n)$ is a point in the case $n=3$.

 The computation of $V^{\it sec}_{p}(n)$ follow the same procedure for every 
 $n$.  For a local scheme $\Gamma$ in $V^{\it loc}_p(n)$ we may set $x_{n}=1$ in 
 the equations. 
 \begin{lemma} A local scheme $\Gamma$ supported at $p$, that belongs to $ V^{\it loc}_p(n)$,  
     is Gorenstein.  The maximal ideal of its affine coordinate ring is spanned by $x_{2},\ldots,x_{n-1},x_{1}$, and its
     socle is generated by $x_{1}$.
     \end{lemma}
     \begin{proof} The scheme $\Gamma$ is Gorenstein by Lemma 
	 \ref{gor}.  The maximal ideal is certainly generated by 
	 $x_{1},x_{2},\ldots,x_{n-1}$, and since $\Gamma$ is 
	 nondegenerate these are linearly independent.  Finally, 
	 $x_{1}x_{i}=0$ for all $i$ by the apolarity condition as soon 
	 as $p\in \Gamma$, so the socle is generated by $x_{1}$.
	 \end{proof}
We may now get explicit equations for $V^{\it loc}_p(n)$. 
If $[\Gamma]\in  V^{\it 
loc}_{p}(n)$, then by definition $m_{p}^{n}=0$.  But the maximal 
ideal is generated by  $x_{1},x_{2},\ldots,x_{n-1}$, so this means that 
any monomial of degree $n$ in the $x_{i}$ must vanish in the coordinate ring of $\Gamma$.

On the other hand, the equations for $\Gamma$ 
 define the products 
 $$x_{i}x_{j}=\sum_{k=2}^{n-1}a_{ijk}x_{k}\quad {\rm and}\quad 
 x_{i}^{2}=x_{1}+\sum_{k=2}^{n-1}a_{iik}x_{k}$$
 in this ring.  Therefore, by iteration, we get polynomial relations in the parameters 
 $a_{ijk}$.   
 Imposing the apolarity condition,
 symmetrizing the parameters and adding the equations for $V^{\it 
 sec}_p(n)$, then after, possibly, saturation we get set theoretic 
 equations for $V^{\it 
 loc}_{p}(n)$.  

When $n=4$ we have the parameters $a_{222},a_{223},a_{233},a_{333}$ 
for $V^{\it sec}_p(4)$ and the relation 
$$a_{223}^2-a_{222}a_{233}+a_{233}^2-a_{223}a_{333}.$$
Thus $V^{\it sec}_p(4)$ is a quadric hypersurface in $\AAA^{4}$.

For $V^{\it loc}_{p}(4)$ we first get the subscheme defined by the 
equations $$x_{1}^2=x_{1}x_{2}=x_{1}x_{3}=0$$
and $$x_{2}^2=x_{1}+a_{222}x_{2}+a_{223}x_{3}, 
x_{2}x_{3}=a_{223}x_{2}+a_{233}x_{3}, 
x_{3}^2=x_{1}+a_{233}x_{2}+a_{333}x_{3}.$$
The coefficient of $x_{1}$ using these relations iteratively to 
compute $x_{2}^{4},\ldots,x_{3}^{4}$, must vanish, so it yields the 
equations
$a_{222}+a_{233}=a_{223}+a_{333}=a_{233}^{2}+a_{223}^{2}=0$. The 
other coefficients give no additional relations, and neither does the 
equations for $V^{\it sec}_p(4)$, so 
$V^{\it loc}_{p}(4)$ is  $1$-dimensional and consists of a pair of affine 
intersecting lines.

When $n=5$, the computation becomes a bit more involved.  There are ten parameters $a_{ijk}$.   
The equations of $V^{\it sec}_p(5)$ are
$$a_{234}^2-a_{233}a_{244}+a_{334}^2-a_{333}a_{344}+a_{344}^2-a_{334}a_{444} =0$$
$$a_{224}a_{234}-a_{223}a_{244}+a_{234}a_{334}-a_{233}a_{344}+a_{244}a_{344}-a_{234}a_{444}=0$$
$$a_{224}a_{233}-a_{223}a_{234}+a_{234}a_{333}-a_{233}a_{334}+a_{244}a_{334}-a_{234}a_{344}=0$$
$$a_{234}^2+a_{224}^2-a_{222}a_{244}+a_{244}^2-a_{223}a_{344}-a_{224}a_{444}=0$$
$$a_{223}a_{224}-a_{222}a_{234}+a_{233}a_{234}+a_{234}a_{244}-a_{223}a_{334}-a_{224}a_{344}=0$$
$$a_{223}^2-a_{222}a_{233}+a_{233}^2-a_{223}a_{333}-a_{224}a_{334}+a_{234}^2=0$$
They define in $\AAA^{10}$ the affine cone over the intersection of 
the Grassmannian variety $\GG(2,5)$ with a quadric.
For $V^{\it loc}_{p}(5)$ there are additional equations defining the cone over 
the tangent developable of a rational normal sextic curve, a codimension $3$ linear section of $V^{\it 
sec}_p(5)$. 
The cone over the rational normal curve parameterizes local apolar subschemes that are not curvilinear.
For the computations in {\it Macaulay2} \cite {MAC2}, see the documented code in \cite{RSMAC}. 

The findings are  summarized in Table 1.

 \begin{table}
\begin{center}
\begin{tabular}{|c|cr|cr|cr|} \hline
$n$  & \multicolumn{2}{c|}{ $V^{\it loc}_{p}(n)$} & \multicolumn{2}{c|}{$V^{\it sec}_{p}(n)$} &
\multicolumn{2}{c|}{$V^{\it aff}_h(n)$}  \\ \hline
 &  dim & degree& dim & degree&  dim & degree \\  \hline

3&  0& 1 &  1& 1 &   3& 1  \\
 &   \multicolumn{2}{c|}{ point} &   \multicolumn{2}{c|}{ $\AAA^{1}$} &    
  \multicolumn{2}{c|}{ $\AAA^{3}$}  \\  \hline

4&  1& 2 &        3     & 2  &  6 &1   \\
&   \multicolumn{2}{c|}{ two lines} & \multicolumn{2}{c|}{ Quadric 
  3-fold}  &   \multicolumn{2}{c|}{ $\AAA^{6}$}   \\

\hline

5& 3 &10 & 6 & 10 &                           10&1    \\
&   \multicolumn{2}{c|}{cone over tangent developable} & 
  \multicolumn{2}{c|}{ cone over $\GG(2,5)\cap Q$} &  
  \multicolumn{2}{c|}{ $\AAA^{10}$}    \\
&    \multicolumn{2}{c|}{of a rational sextic curve} & 
  \multicolumn{2}{c|}{}  &  
  \multicolumn{2}{c|} {}   \\

  \hline

6 &   &&  & & 15   &    12                 \\
 &    &   &  && \multicolumn{2}{c|}{cone over $S_{10}$}   \\  \hline

\end{tabular}
\end{center}

\caption{ }
 \end{table}

\section{Global invariants of $VPS(Q,n)$}\label{sec:geometry}

We consider $VPS(Q,n)$ as a subscheme of $\GG({n-1},{T_{2,q}})$, and the incidence
$$I_{Q}^{VPS}=\{([q'],[\Gamma])|[q']\subset \langle\Gamma\rangle\}\subset \PP(T_{2,q})\times VPS(Q,n).$$   The incidence is a 
projective bundle, $$I_{Q}^{VPS}=\PP(E_{Q})\xrightarrow{\pi} VPS(Q,n),$$
 while the first projection is birational 
 (the rational map $\gamma:\PP(T_{2,q})\dasharrow VPS(Q,n)$ factors through the inverse of this projection).
 Denote by $L$ the tautological divisor on $\PP(E_{Q})$.  It is 
 the pullback of the hyperplane divisor on $\PP(T_{2,q})$.   When 
 $VPS(Q,n)$ is smooth, 
 $$\Pic(I_{Q}^{VPS}) \cong \Pic(VPS(Q,n)) \oplus
\ZZ[L].$$    
Recall from Lemma \ref{plucker}, that the set $H_{h}\subset VPS(Q,n)$ of subschemes $\Gamma$ 
that intersects a hyperplane $h\subset\PP(T_{1})$ form a Pl\"ucker divisor 
restricted to $VPS(Q,n)$.  Therefore the class of the Pl\"ucker divisor 
coincides with the first Chern class $c_{1}(E_{Q})$.

\begin{theorem}\label{PicVPS}
i) $\Pic(VPS(Q,4)\cong \Pic(VPS(Q,5))\cong \ZZ$ 

ii) The ample generator $H$ is very ample, and
$VPS(Q,4)$ and $VPS(Q,5)$ are Fano-manifolds of index 2. 

iii) The boundary in $VPS(Q,n)$ consisting of singular apolar 
subschemes is, when $n\leq 5$, an 
anticanonical divisor.
\end{theorem}

\begin{proof} i)  Let $n=4$ or $n=5$.  Then the Pl\"ucker divisor $H$ is very ample by the 
    above.  Furthermore,  the complement $V^{\it aff}_{p}$ of the 
    special Pl\"ucker divisor defined by a tangent hyperplane to 
    $Q^{-1}\subset \PP(T_{1})$, the divisor 
    $H_{\{x_{n}=0\}}$ in the above notation, is isomorphic to affine space 
    by Proposition 
    \ref{prop:Ln}. Therefore the Picard group has rank $1$ as soon as 
    this special Pl\"ucker divisor is irreducible.
    
    The tangent hyperplanes to $Q^{-1}$ cover all of 
    $\PP(T_{1}),$  so the corresponding Pl\"ucker divisors cover 
    $VPS(Q,n)$. Furthermore, for any subscheme $\Gamma$ in
    $VPS(Q,n)$, there is tangent hyperplane that does not meet 
    $\Gamma$, so these special Pl\"ucker divisors have no 
    common point on $VPS(Q,n)$.   Assume that the special Pl\"ucker 
    divisors are reducible, then we may write $H=H_{1}+H_{2}$, where 
    both $H_{1}$ and $H_{2}$ moves without base points on $VPS(Q,n)$. Since 
    $H\cdot l=1$ for every line on
    $VPS(Q,n)$, only one of the two components can 
    have positive intersection with a line.  The other, say $H_{2}$, 
    must therefore contain 
    every line that it intersects. But this is impossible, since $H_2$ must contain all of $VPS(Q,n)$, by the following lemma:
     \begin{lemma}\label{lineconnected}  Any two polar simplices $\Gamma$ and $\Gamma'$ are connected by a sequence of lines in $VPS(Q,n)$.
     \end{lemma}
\begin{proof} This is immediate when $n=2$.   For $n>2$, let $[l]\in \Gamma$ and $[l']\in \Gamma'$, and let $\PP(U)=h_l\cap h_{l'}\subset \PP(T_1)$ be the intersection of their polar hyperplanes.  Then $q=l^2+l^2_1+q_U=(l')^2+(l'_1)^2+q_U$ for $q_U\in U^2$ and suitable $l_1$ and $l'_1$.  Let $\Gamma_U$ be a polar simplex for $q_U$.  Then $\Gamma$ is line connected to $\Gamma_U\cup \{[l_1],[l]\}$ by induction hypothesis.  Likewise   $\Gamma$ is line connected to $\Gamma_U\cup \{[l'_1],[l']\}$.  Finally $\Gamma_U\cup \{[l_1],[l]\}$ and $\Gamma_U\cup \{[l'_1],[l']\}$ span a line in $VPS(Q,n)$, which completes the induction.
  \end{proof}
    

ii) 
Since $\Pic(I_{Q}^{VPS}) \cong \Pic(VPS(Q,n)) \oplus
\ZZ[ L]$
we deduce from i) that the birational morphism
$$\sigma\colon I_{Q}^{VPS} \to \PP(T_{2,q})$$
has an irreducible exceptional divisor.
Let $E \in \Pic(I_{Q}^{VPS})$ be the class of this exceptional
divisor.  Then, since the map $\gamma: \PP(T_{2,q})\dasharrow 
\GG({n-1},{T_{2,q}})$ is defined by polynomials of degree
$\binom{n}{2}$, the size 
of the minors in the Mukai form, we have
$$\pi^*H = \binom{n}{2} L -E\;{\rm and}\; K_{I_{Q}^{VPS}}=-(\binom{n+1}{2}-1)L+E.$$
On the other hand
$H = c_1(E_{Q})$
where $I_{Q}^{VPS}=\PP(E_{Q})$ is a projective bundle over $VPS(Q,n)$
so 
$$-(\binom{n+1}{2}-1)L+E=K_{I_{Q}^{VPS}} =
\pi^*K_{VPS}+\pi^*(c_1(E_{Q})) -(n-1)L.
$$
Therefore $-K_{VPS(Q,n)} = 2H$.  Finally, since  
$VPS(Q,n)\subset\GG({n-1},{T_{2,q}})$ contains lines, $H$ is not divisible.

iii)  The boundary in $VPS(Q,n)$ consisting of singular apolar 
subschemes, coincides, by Lemma \ref{nonred}, with the set of subschemes $\Gamma\subset 
\PP(T_{1})$ that 
intersect quadric $Q^{-1}$.  The Pl\"ucker divisor $H$ is represented by the 
divisor of subschemes $\Gamma$ that intersect a hyperplane in 
$\PP(T_{1})$, so $-K=2H$ is represented by the boundary.
\end{proof}




\begin{theorem}\label{degree}
    Let $n>2$ and let $VPS(Q,n)\subset \GG({n-1},{T_{2,q}})$ be the 
    variety of polar simplices in its Grassmannian embedding, with 
    Pl\"ucker divisor $H$.  The $VPS(Q,n)$ has degree   
$$
H^{m} = \sum_{\lambda \vdash m} \binom{m}{\lambda} /
({\lambda^* !}) \cdot d_\lambda
$$
where the sum runs over all partitions 
$\lambda = (\lambda_1, \ldots ,
\lambda_{n})$ of $m=\binom{n}{2} ={\rm dim} VPS(Q,n)$ into integers $ 
n-1 \ge \lambda_1 \ge \ldots \ge
\lambda_{n} \ge 0 $. Here $\lambda^*=(\lambda^{*}_1, \ldots ,
\lambda^{*}_{n-1})$ denotes the sequence
$\lambda^*_i = \mid \{ j \mid \lambda_j = i \} \mid $ and 
$\lambda^{*}!=\Pi \lambda^*_i!$. Finally 
$$d_\lambda = \prod_{1\le i < j \le n}(D_i+D_j)$$
is the intersection number of $m$ divisors on the product
$$
\PP^{n-1-\lambda_1} \times \ldots \times
\PP^{n-1-\lambda_{n}}$$ 
with $D_i$ the pullback of the hyperplane class
on the $i^{th}$ component.
\end{theorem}

\begin{proof}
We first show that for $\binom{n}{2}$ general hyperplanes 
$h_{i}\subset \PP(T_{1})$, the corresponding 
Pl\"ucker divisors $H_{h_{i}}$ has a proper transverse intersection on the smooth part of $VPS(Q,n)$.
Therefore, by properness, the intersection is finite, and, by transversality, it is smooth,
so it is a finite set of points.  The cardinality is the degree of $VPS(Q,n)$.

First, let  $\lambda = (\lambda_1, \ldots ,
\lambda_{n})$ be a partition of $m$ and
consider the partition 
$h_{11},\ldots,h_{1\lambda_{1}},\ldots,h_{n1},\ldots,h_{n\lambda_{n}}$ of $m$ general hyperplanes 
into $n$ sets of size $\lambda_{1},\ldots,\lambda_{n}$. 
Let $L_{i}=\cap_{j} h_{ij}$, it is a linear space of dimension 
$n-1-\lambda_{i}$.
Consider the product of these linear spaces in the product 
$\PP(T_{1})^{n}$:
$$L_{1}\times \cdots \times L_{n}\subset\PP(T_{1})\times \cdots\times \PP(T_{1}).$$ Let 
$\Delta\in \PP(T_{1})^{n}$ be the union of all diagonals and let $L^{o}=L_{1}\times\cdots\times 
L_{n}\setminus \Delta\subset \PP(T_{1})^{n}$. Then $L^{o}$ parameterizes 
$n$-tuples of points $\Gamma=\{p_{1},\ldots,p_{n}\}\subset \PP(T_{1})$, with $p_{i}\in L_{i}$. 
Of course, $L^{o}$ has a natural map to the Hilbert scheme of 
$\PP(T_{1})$ that forgets the ordering, so we will identify elements in $L^{o}$ with their 
image in the Hilbert scheme.


    
    Consider the incidence between subschemes $\Gamma\in L^{o}$ and 
    quadratic forms $q\in T_{2}$:
    $$I_{L}=\{(\Gamma,[q])|I_{\Gamma}\subset q^{\bot}\}\subset L^{o}\times 
   \PP(T_{2})$$
   This variety is defined by the equations $h_{ij}(p_{i})=0$ and the 
   apolarity, $q(I_{\Gamma})=0$. 
   
   Clearly $L$ is a smooth scheme of dimension $\binom{n}{2}$.  
   The fibers of the projection $I_{L}\to L$ are $(n-1)$-dimensional 
   projective spaces, so $I_{L}$ is a smooth variety of dimension 
   equal to dim$\PP(T_{2})$.  The projection $I_{L}\to \PP(T_{2})$ is 
   clearly onto, so the fibers are finite.  Since both spaces are 
   smooth, the general fiber is smooth. 
   Now, $\Gamma\subset L^{o}$ lies in the fiber over $[q]$, precisely 
   when $I_{\Gamma}\subset q^{\bot}$, i.e. $[\Gamma]\in VPS(Q,n)$ and 
   $h_{ij}(p_{i})=0$, i.e. $[\Gamma]$ lies in the intersection of all 
   the Pl\"ucker hyperplanes $H_{h_{ij}}$.
  Since the general fibers are smooth  
  the divisors $H_{h_{ij}}$ intersect transversally in $VPS(Q,n)$, 
  where $Q=\{q=0\}$, and 
  have an isolated intersection point at each point $[\Gamma]$.  
  Turning the argument around and considering all partitions, we get 
  that for general hyperplanes $h_{1},\ldots,h_{m}$ in $\PP(T_{1})$ 
  the Pl\"ucker hyperplanes $H_{h_{i}}$ 
  has a transversal intersection at a finite number of points in $VPS(Q,n)$ 
  corresponding to smooth apolar subschemes. 
  
We proceed to compute the cardinality of the intersection, i.e. the formula given in the theorem.
Let $[\Gamma] =[\{p_1,\dots,p_n\}]\in VPS(Q,n)$ be a point in the intersection of the hyperplanes 
$H_{h_j}$.
Then each $h_{j}$ contains some $p_{i}\in \Gamma$, by the definition of $H_{h_{j}}$.
For each $i$ let $\lambda_{i}$ be the number of hyperplanes $h_{j}$ that contains $p_{i}$.  
 The set of positive integers $\{\lambda_{1},...,\lambda_{n}\}$ must add up to $m$:  It is at least $m$ by definition, and at most $m$ by the generality assumption discussed above.  Therefore the point $[\Gamma]$ defines a unique
 partition of the set of hyperplanes $\{h_{j}\}_{j=1}^m$ into subsets $\{h_{ij}\}_{j=1}^{\lambda_i}$ of cardinality $\lambda_{i}$, as above. 
 
  The factor $\binom{m}{\lambda}$ in the degree formula counts the number of ordered partitions of $m$ hyperplanes into subsets of cardinality $\lambda_i$, while ${\lambda^* !}$ counts the permutations of the subsets of the same cardinality, i.e. the number of ordered partitions determined by $[\Gamma]$.  Therefore the remaining factor $d_\lambda$ for each partition should count the number of polar simplices $\Gamma$ that intersect the $n$ linear subspaces $L_i=\cap_jh_{ij}\subset \PP(T_1)$
of codimension $\lambda_{i}, i=1,...,n$. 

Let $[\Gamma]=[\{p_1,...,p_n\}]\in VPS(Q,n)$ and assume that $(p_1,...,p_n)\in L_1\times\cdots\times L_n$.  For each pair of linear spaces $L_{i},L_{j}$
 the bilinear form associated to the quadratic form $q$ restricts to a linear form on the product $L_{i}\times L_{j}$ that vanishes on $(p_{i},p_{j})$.
  This linear form defines a divisor $H_{ij}$ in the divisor class $H_{i}+H_{j}$, where $H_{i}$ is the pullback to the product of the hyperplane class on $L_{i}$. 
   If $D_{ij}$ and $D_{i}$ are the pullbacks of $H_{ij}$, respectively $H_{i}$, to the product $\Pi_{i}L_{i}$, then $\Gamma\subset \Pi_{i}L_{i} $ lies in the intersection $\cap_{i<j}D_{ij}$.  
   
   Conversely, consider a point $(p_{1},\ldots,p_{n})\in L_{1}\times \cdots \times L_{n}$
  that lies in the intersection of the divisors $\cap_{i<j}D_{ij}$.  
  The projection of this point into $\PP(T_{1})$ is a collection of 
  $n$ points $\Gamma=\{p_{1},\ldots,p_{n}\}$.  Let $p_{i}=[v_{i}], 
  v_{i}\in T_{1}$, then the 
  bilinear form $q: T_{1}\times T_{1}\to \CC,\quad 
  (q^{-1}(v_{i})(q))(v_{j})=0$  for every $i\not=j$, so the hyperplanes 
  $p_{i}^{\bot}\subset \PP(S_{1})$ form a polar simplex to $Q$.  
  Hence $[\Gamma]$ is a point in $VPS(Q,n)$.

Thus $d_\lambda$ counts the number of polar simplices $\Gamma$ that intersect the $n$ linear subspaces $L_i$ and the degree formula follows.

\end{proof}
The Theorem \ref{main} in the introduction  follows from Corollary \ref{cor:tangentdimension},  Corollary \ref{sing} and Theorem \ref{PicVPS}.
 Theorem \ref{main2} follows from Corollary \ref{grassmannian embedding}, Corollary \ref{TQ}, Theorem \ref{PicVPS} and the degree is computed from Theorem \ref{degree}.
Theorem \ref{main3} follows from Theorem \ref{thm:linear section}, Corollary \ref{cor:n=6} and Corollary \ref{VAPS}.

\end{document}